\documentclass{amsproc}
\usepackage{amssymb}

\usepackage{amsmath}
\usepackage{xspace,xcolor}
\usepackage[breaklinks,colorlinks,citecolor=teal,linkcolor=teal,urlcolor=teal,hyperindex]{hyperref} 
\usepackage[alphabetic]{amsrefs}
\usepackage[cmtip,all]{xy}

\newtheorem{theorem}{Theorem}[section]
\newtheorem{lemma}[theorem]{Lemma}
\newtheorem{proposition}[theorem]{Proposition}
\newtheorem{corollary}[theorem]{Corollary}

\theoremstyle{definition}
\newtheorem{definition}[theorem]{Definition}

\newtheorem{question}[theorem]{Question}

\theoremstyle{remark}
\newtheorem{example}[theorem]{Example}
\newtheorem{remark}[theorem]{Remark}
\newtheorem*{acknowledgments}{Acknowledgments}

\newcommand{\Mod}[1]{\ (\text{mod}\ #1)}

\def\A{{\mathbb A}}
\def\C{{\mathbb C}}
\def\DD{{\mathbb D}}

\def\N{{\mathbb N}}

\def\Q{{\mathbb Q}}
\def\R{{\mathbb R}}

\def\Z{{\mathbb Z}}

\def\O{\mathcal{O}}

\def\fa{\mathfrak{a}}

\def\fm{\mathfrak{m}}

\def\a{\alpha}
\def\b{\beta}
\def\g{\gamma}

\def\f{\phi}
\def\ff{\psi}
\def\e{\eta}
\def\ep{\epsilon}

\def\m{\mu}

\def\p{\pi}

\def\s{\sigma}
\def\t{\tau}

\def\D{\Delta}
\def\G{\Gamma}

\def\Om{\Omega}

\def\.{\cdot}
\def\~{\widetilde}
\def\o{\circ}
\def\ov{\overline}

\def\rat{\dashrightarrow}

\def\({\left(}
\def\){\right)}

\def\liminv{\underleftarrow{\lim}}

\renewcommand{\and}{ \ \ \text{ and } \ \ }
\renewcommand{\for}{ \ \ \text{ for } \ \ }

\def\reg{\mathrm{reg}}
\def\sm{\mathrm{sm}}
\def\main{\mathrm{main}}
\def\sing{\mathrm{sing}}
\def\red{\mathrm{red}}
\def\an{\mathrm{an}}

\def\Jac{\mathrm{Jac}}

\DeclareMathOperator{\codim} {codim}

\DeclareMathOperator{\Spec} {Spec}

\DeclareMathOperator{\Bl} {Bl}
\DeclareMathOperator{\val} {val}

\DeclareMathOperator{\Ex} {Ex}
\DeclareMathOperator{\mult} {mult}
\DeclareMathOperator{\ord} {ord}

\DeclareMathOperator{\Supp} {Supp}

\DeclareMathOperator{\Fitt} {Fitt}

\DeclareMathOperator{\Hom} {Hom}

\begin{document}

\title{The space of arcs of an algebraic variety}

\author{Tommaso de Fernex}
\address{Department of Mathematics, University of Utah, Salt Lake City, UT 48112, USA}
\email{{\tt defernex@math.utah.edu}}
\thanks{The research was partially supported by 
NSF grant DMS-1402907 and NSF FRG grant DMS-1265285.}

\dedicatory{In memory of John Forbes Nash, Jr.}

\subjclass[2010]{Primary 14E18; Secondary 14J17}

\date{}

\begin{abstract}
The paper surveys several results on the topology of the space of arcs
of an algebraic variety and the Nash problem on the 
arc structure of singularities. 
\end{abstract}

\maketitle

\section{Introduction}

In 1968, Nash wrote a paper on the arc structure of singularities of complex algebraic varieties \cite{Nas95}.
While the paper was only published many years later, 
its content was promoted by Hironaka and later by Lejeune-Jalabert, and it is thanks
to them that the mathematical community came to know about it.

In that paper, the space of arcs of an algebraic variety is regarded for the first time
as the subject of investigation in its own right.
Other papers which appeared in those years and provide insight to the topological structure 
of spaces of arcs are \cite{Gre66,Kol73}.
Since then, spaces of arcs have become a central object of study by algebraic geometers. 
They provide the underlying space in motivic integration, where arcs take the role of 
the $p$-adic integers in $p$-adic integration \cite{Kon95,DL99,Bat99}. 
The relationship between constructible sets in arc spaces
and invariants of singularities in the minimal model program
has yielded important applications in birational geometry (e.g., \cite{Mus01,Mus02,EMY03,EM04}).

Nash viewed the space of arcs
as a tool to study singularities of complex algebraic varieties, and
for this reason he focused on the set of arcs on a variety that
originate from the singular points. 
Nash realized that there is a close connection between the families of arcs
through the singularities and certain data associated
with resolutions of singularities whose existence had just been 
established a few years earlier \cite{Hir64}.
He gave a precise formulation predicting that such families of arcs 
should correspond to those exceptional divisors that are ``essential'' 
for all resolutions of singularities. Establishing this correspondence 
became known as the {\it Nash problem}.

The Nash problem has remained wide open until recently and is still not completely understood. 
This note focuses on this problem and the progress made around it in recent years.
The correspondence proposed by Nash has been shown to hold in dimension two \cite{FdBPP12}
and to fail, in general, in all higher dimensions \cite{IK03,dF13,JK13}. 
However, this should not be viewed as the end of the 
story, but rather as an indication of the difficulty of the problem. Partial results
have been obtained in higher dimensions (e.g., \cite{IK03,dFD16}), and a complete solution
of the problem will only be reached once the correct formulation is found.

\begin{acknowledgments}
We thank Roi Docampo, Javier Fern\'andez de Bobadilla, Shihoko Ishii, J\'anos Koll\'ar,
Ana Reguera, and Wim Veys for many valuable comments and suggestions. 
We would like to thank the referees for their careful reading of the paper
and for their corrections and valuable remarks.
\end{acknowledgments}

\section{The space of arcs}

We begin this section with a quick overview of the definition of arc space, 
referring to more in depth references such as \cite{DL99,Voj07,EM09} for further details. 

Let $X$ be a scheme of finite type over a field $k$. 
The space of arcs $X_\infty$ of $X$ is a scheme whose $K$-valued
points, for any field extension $K/k$, are formal arcs
\[
\a \colon \Spec K[[t]] \to X.
\]
It is constructed as the inverse limit of the jet schemes $X_m$ of $X$, 
which parameterize jets $\g \colon \Spec K[t]/(t^{m+1}) \to X$.
These schemes are defined as follows. 

\begin{definition}
For every $m \in \N$, the \emph{$m$-th jet scheme} $X_m$ of $X$
is the scheme representing the functor that takes a $k$-algebra $A$ to
the set of $A$-valued $m$-jets
\[
X(A[t]/(t^{m+1})) := \Hom_k(\Spec A[t]/(t^{m+1}),X). 
\]
\end{definition}

For every $p \ge m$ there is a natural projection map $X_{p} \to X_m$ 
induced by the truncation homorphism $A[t]/(t^{p+1}) \to A[t]/(t^{m+1})$. 
These projections are affine, and hence one can take the inverse limit
of the corresponding projective system in the category of schemes.

\begin{definition}
The \emph{space of arcs} (or \emph{arc space}) of $X$ is the inverse limit
\[
X_\infty := \liminv X_m. 
\]
\end{definition}

The next property is not a formal consequence of the definition, and 
the proof uses methods of derived algebraic geometry.
We do not know if there is a more direct proof. 

\begin{theorem}[Bhatt \protect{\cite[Corollary~1.2]{Bha}}]
The arc space $X_\infty$ represents the functor that takes a $k$-algebra $A$ to
the set of $A$-valued arcs
\[
X(A[[t]]) := \Hom_k(\Spec A[[t]],X). 
\]
\end{theorem}

The truncations $A[[t]] \to A[t]/(t^{m+1})$ induce
projection maps $\p_{X,m} \colon X_\infty \to X_m$. Taking $m=0$, we obtain the projection
\[
\p_X \colon X_\infty \to X
\]
which maps an arc $\a(t) \in X_\infty(K)$ to the point $\a(0) \in X(K)$ 
where the arc stems from.\footnote{The notation $\a(t)$ refers to the fact that, 
in local parameters of $X$ at $\a(0)$, 
the arc is given by formal power series in $t$.}

\begin{remark}
Arc spaces are closely connected to valuation theory (cf.\ \cite[Theorem~1.10]{Reg95}; see also \cite{Ple05}). 
For any given field extension $K/k$, an arc $\a \colon \Spec K[[t]] \to X$ defines a valuation
\[
\val_\a \colon \O_{X,\a(0)} \to \N \cup \{\infty\}.
\]
given by $\val_\a(h) := \ord_t(\a^\sharp h)$. 
Here $0$ denotes the closed point of $\Spec K[[t]]$. We will denote by $\e$ the generic point. 
If $X$ is a variety and $\a(\e)$ is the generic point of $X$, then $\val_\a$ extends to a valuation
\[
\val_\a \colon k(X)^* \to \Z.
\]
For every irreducible closed set $C \subset X_\infty$,
we denote by $\val_C$ the valuation defined by the generic point of $C$. 
We will use several times the fact that 
if an arc $\a$ is a specialization of another arc $\b$, then $\val_\a(h) \ge \val_\b(h)$
for every $h \in \O_{X,\a(0)}$, which can be easily seen
by observing that, writing $\a^\sharp h = \sum a_it^i$ and $\b^\sharp h = \sum b_it^i$, 
each coefficient $a_i$ is a specialization of the corresponding coefficient $b_i$.
\end{remark}

The space of arcs of an affine space $\A^n$ is easy to describe.
For every $i \ge 0$, we introduce $n$-ples of variables 
$x_1^{(i)},\dots,x_n^{(i)}$. We identify $x_j = x_j^{(0)}$ and
write for short $x_j' = x_j^{(1)}$ and $x_j'' = x_j^{(2)}$. 
The arc space of $\A^n$ is the infinite dimensional affine space
\[
(\A^n)_\infty = \Spec
k[x_j,x_j',x_j'',\dots]_{1 \le j \le n},
\]
where a $K$-valued point $(a_j,a_j', a_j'',\dots)_{1 \le j \le n}$ 
corresponds to the $K$-valued arc $\a(t)$ with components
$x_j(t) = a_j + a_j't + a_j''t^2 + \dots $.

If $X$ is an affine scheme, defined by equations $f_i(x_1,\dots,x_n) = 0$
in an affine space $\A^n$, then $X_\infty$ parameterizes $n$-ples of 
formal power series $(x_1(t),\dots,x_n(t))$ subject to the conditions
$f_i(x_1(t),\dots,x_n(t)) = 0$ for all $i$. These conditions describe $X_\infty$ 
as a subscheme in an infinite dimensional affine space
defined by infinitely many equations in infinitely many variables.
The equations of $X_\infty$ in $(\A^n)_\infty$ can be generated using 
Hasse--Schmidt derivations \cite{Voj07}. 
There is a sequence $(D_0,D_1,D_2,\dots)$ of $k$-linear maps 
\[
D_i \colon k[x_1,\dots,x_n] \to k[x_j,x_j',x_j'',\dots]_{1 \le j \le n}
\]
uniquely determined by the conditions 
\[
D_i(x_j) = x_j^{(i)}, \quad D_k(fg) = \sum_{i+j = k} D_i(f)D_j(g),
\]
and the ideal of $X_\infty$ in $(\A^n)_\infty$ is 
generated by all the derivations $D_i(f_j)$, $i \ge 0$, of
a set of generators $f_j$ of the ideal of $X$ in $\A^n$. 

The arc space of an arbitrary scheme $X$ can be glued together, scheme theoretically, 
from the arc spaces of its affine charts. 
The Zariski topology of the arc space agrees with the inverse limit topology. 
Excluding of course the trivial case where $X$ is zero dimensional, $X_\infty$ 
is not Noetherian and is not a scheme of finite type. Yet, some finiteness is built into it. 

\medskip

We henceforth assume the following:
\begin{quote}
\emph{$X$ is a variety defined over an algebraically closed field $k$ of characteristic zero.}
\end{quote}
We will be working in this setting throughout the paper until the
last section where varieties over fields of positive characteristics will be considered. 

If $X$ is a smooth $n$-dimensional variety, then each jet scheme $X_m$ is smooth
and $X_\infty$ is the inverse limit of a system of
locally trivial $\A^n$ fibrations $X_{m+1} \to X_m$.
This can be seen by reduction to the case of an affine space using Noether normalization, 
or equivalently by Hensel's lemma. 
It follows in this case that $X_\infty$ is an integral scheme
and the projections $X_\infty \to X_m$ are surjective. 
Furthermore, $\p_X^{-1}(S)$ is irreducible for any irreducible set $S \subset X$. 

\begin{remark}
The first jet scheme $X_1$ of a smooth variety $X$ is the same as the tangent bundle of $X$.
However, for any $m \ge 1$ the fibration $X_{m+1} \to X_m$ does not have a natural 
structure of vector bundle. For example, the nonlinear change of coordinates
$(u,v) = (x+y^2,y)$ on $X = \Spec k[x,y]$ induces the affine change of
coordinates $(u'',v'') = (x'' + b^2,y'')$ on 
the fiber of $X_2 \to X_1$ over a point $(0,0,a,b) \in X_1 = \Spec [x,y,x',y']$.
In general, for every $m$ there is a natural section $X \to X_m$ which takes a point of $X$
to the constant $m$-jet through that point, but there
is no natural section $X_m \to X_p$ for $p > m > 0$. 
\end{remark}

If $X$ is singular, then the jet schemes $X_m$ can have 
several irreducible components and non-reduced structure, 
the maps $X_{m+1} \to X_m$ fail to be surjective, there are jumps in their fiber dimensions, 
and the inverse image $\p_X^{-1}(S)$ of an 
irreducible set $S \subset X$ may fail to be irreducible. 
These pathologies make the study of $X_\infty$ a difficult task. 

The systematic study of the space of arcs 
began in the sixties through the works of Greenberg, Nash, and Kolchin. 
In this context, Greenberg's approximation theorem gives the following property. 

\begin{theorem}[Greenberg \protect{\cite[Theorem~1]{Gre66}}]
\label{t:Gre}
For any system of polynomials $f_1,\dots,f_r \in k[x_1,\dots,x_n]$
there are numbers $N, c \ge 1$ and $s \ge 0$ 
such that for any $m \ge N$ and every $x(t) \in k[[t]]^n$ such that 
$f_i(x(t)) \equiv 0 \Mod{t^m}$, 
there exists $y(t) \in k[[t]]^n$ such that 
$y(t) \equiv x(t) \Mod{t^{[m/c]-s}}$ and $f_i(y(t)) = 0$.
\end{theorem}

The image of $X_\infty$ in $X_m$ is the intersection of the images of
the jet schemes $X_p$ for $p \ge m$, which form a nested sequence of constructible sets.
The content of Greenberg's theorem is that the sequence stabilizes, which means that
the image of $X_\infty$ agrees with the image of $X_p$ for $p \gg m$. 
It follows in particular that
for every $m$ the image of $X_\infty$ in $X_m$ is constructible.
This can be viewed as the first structural result on arc spaces.

Following the terminology of \cite[Definition~(9.1.2)]{EGAiii_1}, a \emph{constructible} set in $X_\infty$
is, by definition, 
a finite union of finite intersections of retrocompact open sets and their 
complements, where a subset $Z \subset X_\infty$ is said to be \emph{retrocompact}
if for every quasi-compact open set $U \subset X_\infty$, the intersection $Z \cap U$
is quasi-compact.\footnote{According to this definition, 
a closed subset of $X_\infty$ needs not be constructible. For instance if $Z \subset X$ is a proper closed
subscheme, then $Z_\infty$ is closed in $X_\infty$ but is not constructible.}
This means that a subset $C \subset X_\infty$ is constructible if and only if
it is the (reduced) inverse image of a constructible set on some finite level $X_m$
\cite[Th\'eor\`eme~(8.3.11)]{EGAiv_3}.
Such sets are nowadays commonly called \emph{cylinders}. 
Theorem~\ref{t:Gre} implies the following property.\footnote{The property 
can also be viewed as a consequence of Pas' quantifier elimination theorem \cite{Pas89}.}

\begin{corollary}
The image at any finite level $X_m$ of a constructible subset of $X_\infty$ is constructible. 
\end{corollary}

The second theorem we want to review is Kolchin's irreducibility theorem. 
Since the proof goes in the direction of the main focus of this paper, we outline it. 
The proof given here, which is taken from \cite{EM09}, is
different from the original proof of Kolchin. 
Other proofs of this property can be found in \cite{Gil02,IK03,NS10}. 

\begin{theorem}[Kolchin \protect{\cite[Chapter~IV, Proposition~10]{Kol73}}]
\label{t:Kolchin}
The arc space $X_\infty$ of a variety $X$ is irreducible. 
\end{theorem}

\begin{proof}
Let $f \colon Y \to X$ be a resolution of singularities. Let $Z \subset X$ be the indeterminacy
locus of $f^{-1}$ and $E = f^{-1}(Z)_\red$ be the exceptional locus of $f$. Since $Y$ is smooth, 
its arc space $Y_\infty$ is irreducible. It is therefore sufficient to show that 
the induced map 
\[
f_\infty \colon Y_\infty \to X_\infty
\]
is dominant. Since $f$ is an isomorphism in $f^{-1}(X\setminus Z)$, 
the valuative criterion of properness implies that every arc $\a$ on $X$
that is not entirely contained in $Z$ lifts to $Y$. 
If $Z = \bigcup Z_i$ is the decomposition into irreducible components, then
$Z_\infty = \bigcup (Z_i)_\infty$, set theoretically. By induction on dimension, each
$(Z_i)_\infty$ is irreducible, and therefore if $U_i \subset Z_i$ is a dense open subset then
$(U_i)_\infty$ is dense in $(Z_i)_\infty$. By generic smoothness, we can find 
a dense open subset $U_i \subset Z_i$ and an open set $V_i \subset E$ such that $f$ 
restricts to a smooth map $V_i \to U_i$. Every arc on $U_i$ lifts to $V_i$, and hence
$(U_i)_\infty \subset f_\infty((V_i)_\infty)$. Therefore each $(Z_i)_\infty$
is in the closure of $f_\infty(Y_\infty)$. This shows that $f_\infty$ is dominant, 
and the theorem follows. 
\end{proof}

The next example shows that $f_\infty$ needs not be surjective. 

\begin{example}
\label{eg:1}
Let $X \subset \A^3$ be the Withney umbrella, defined by the equation $xy^2 = z^2$. 
Its singular locus $X_\sing$ is the $x$-axis, and the normalization $Y \to X$
gives a resolution of singularities. The exceptional divisor $E \subset Y$ maps 
generically two-to-one over $X_\sing$ with ramification at the origin. 
It follows that for every power series $x(t) \in k[[t]]$ with $\ord_t(x(t)) = 1$,  
the arc $\a = (x(t),0,0) \in X_\infty$, which is a smooth arc on $X_\sing$
passing through the origin, cannot lift to $E$ and hence is not in $f_\infty(Y_\infty)$. 
\end{example}

More information on the structure of the arc space of a singular variety $X$ can be 
obtained by a careful analysis of the truncation maps $X_m \to X_n$, defined for $m > n$, 
and the maps $f_m \colon Y_m \to X_m$ induced by a resolution
of singularities $f \colon Y \to X$. Understanding these maps is a delicate but rewarding task. 
Both sets of maps were studied by Denef and Loeser in connection to motivic integration \cite{DL99}, 
and their description plays a key role in relating the geometry of arc spaces 
to invariants of singularities in the minimal model program.

A consequence of one of the results of \cite{DL99} 
is that images of many constructible sets in the arc space $Y_\infty$ of a resolution $Y \to X$
are not far from being constructible in $X_\infty$. 

\begin{theorem}
\label{t:f-infty-constr}
Let $f \colon Y \to X$ be a resolution of singularities 
and let $C \subset Y_\infty$ be a constructible set.
Assume that none of the irreducible components of $\ov C$
is contained in the arc space of the exceptional locus $\Ex(f)$ of $f$. 
Then there is a constructible set $D \subset X_\infty$ such that
$D \subset f_\infty(C) \subset \ov D$. 
\end{theorem}

\begin{proof}
Without loss of generality, we can assume that $\ov C$ is irreducible. 
Then there is a constructible set $S \subset Y_p$, for some $p \ge 0$, such that 
$\ov S$ is irreducible and $C = \p_{Y,p}^{-1}(S)$,
where $\p_{Y,p} \colon Y_\infty \to Y_p$ is the truncation map. 
Let $\Jac_f :=\Fitt^0(\Om_{Y/X})$ denote the Jacobian ideal sheaf of $f$.   
Since $C \not\subset \Ex(f)_\infty$, we have $e := \val_C(\Jac_f) < \infty$.
By replacing $p$ with a larger integer, we can assume that $p \ge 2e$
and $C  = \p_{Y,p-e}^{-1}(\p_{Y,p-e}(C))$. 

There is a dense relatively open subset $S^\o \subset S$ such that, 
letting $C^\o := \p_{Y,p}^{-1}(S^\o)$ we have $\val_\a(\Jac_f) = e$ for all $\a \in C^\o$. 
Note that $C^\o$ is dense in $C$ and $\p_{Y,m}(C^\o)$ is constructible in $Y_m$ for every $m$.  

By \cite[Lemma~3.4]{DL99} (see also (a') in the proof), 
for every $m \ge p$ the fiber $F$ of $f_m \colon Y_m \to X_m$
through a point of $\p_{Y,m}(C^\o)$ is an affine space of dimension $e$ which
is contained in a fiber of the projection $Y_m \to Y_{m-e}$. 
This implies that $F$ is entirely contained in $\p_{Y,m}(C^\o)$, and therefore we have
\[
f_m^{-1}\big(f_m(\p_{Y,m}(C^\o))\big) = \p_{Y,m}(C^\o).
\]
Using the commutativity of the diagram
\[
\xymatrix{
Y_\infty \ar[r]^{f_\infty}\ar[d]_{\p_{Y,m}} & X_\infty \ar[d]^{\p_{X,m}}\\
Y_m \ar[r]^{f_m} & X_m
}
\]
we see that $f_\infty(C^\o) = \p_{X,m}^{-1}\big(f_m(\p_{Y,m}(C^\o))\big)$.
Note that this is a constructible set in $X_\infty$ since $f_m(\p_{Y,m}(C^\o))$
is constructible in $X_m$. 
As we have 
\[
f_\infty(C^\o) \subset f_\infty\big(\ov{C^\o}\big) \subset \ov{f_\infty(C^\o)}
\]
and $\ov{C^\o} = C$, we can take $D := f_\infty(C^\o)$. 
\end{proof}

\section{Arcs through the singular locus}
\label{s:Nash}

Let $X$ be a variety defined over an algebraically closed field $k$ of characteristic zero. 

The main contribution of \cite{Nas95} is the realization that, on $X$, there are
only finitely many maximal families of arcs through the singularities,
that is to say that the set $\p_X^{-1}(X_\sing) \subset X_\infty$ 
has finitely many irreducible components. 
Moreover, each such family corresponds to a specific component
of the inverse image of $X_\sing$ on a resolution of singularities of $X$. 

This property follows by a variant of the proof of Kolchin's theorem given in the previous section.
We should stress that Nash's result predates Kolchin's theorem. 
The argument goes as follows. 

Let $f \colon Y \to X$ be a resolution of singularities, 
and let 
\[
f^{-1}(X_\sing)_\red = \bigcup_{i \in I} E_i
\] 
be the decomposition into irreducible components.
The set $I$ is finite because $Y$ is Noetherian, and each
$\p_Y^{-1}(E_i)$ is irreducible because $Y$ is smooth. 
Arguing as in the proof of Theorem~\ref{t:Kolchin}, one deduces 
that the map $f_\infty$ restricts to a dominant map 
$\p_Y^{-1}(f^{-1}(X_\sing)) \to \p_X^{-1}(X_\sing)$, 
and therefore there is a finite decomposition into irreducible components
\[
\p_X^{-1}(X_\sing)_\red = \bigcup_{i \in I} C_i, \quad\text{where}\quad 
C_i := \ov{f_\infty(\p_Y^{-1}(E_i))} \subset X_\infty.
\]
Let $J \subset I$ be the set of indices $j$ for which $C_j$ is 
a maximal element of $\{ C_i \}_{i\in I}$, where maximality is intended with respect
to inclusions.

It is convenient at this point to introduce the following notation. 

\begin{definition}
\label{d:max-div-set}
The \emph{maximal divisorial set} associated to a prime divisor $E$ on a resolution $Y$ over $X$
is the set
\[
C_X(E) := \ov{f_\infty(\p_Y^{-1}(E))} \subset X_\infty.
\] 
\end{definition}

\begin{remark}
The definition of $C_X(E)$ does not require the existence of a resolution. 
Only assuming that $Y$ is normal one can take 
$C_X(E) := \ov{f_\infty(\p_Y^{-1}(E \cap Y_\sm))}$.
In this way, the definition extends to positive characteristics.
\end{remark}

We can then state Nash's result as follows. 

\begin{theorem}[Nash \protect{\cite[Propositions~1 and~2]{Nas95}}]
\label{t:Nash}
The set of arcs through the singular locus $X_\sing$ of a variety $X$
has a decomposition into finitely many irreducible components given by
\[
\p_X^{-1}(X_\sing)_\red = \bigcup_{j \in J} C_X(E_j)
\]
where each $E_j$ is a prime divisor over $X$. 
\end{theorem}

\begin{remark}
To be precise, in \cite{Nas95} arcs are assumed to be defined by converging power series.
If $X$ is a complex variety, up to rescaling of the parameter, any such arc is given by a homolorphic map
$\a \colon \DD \to X$, where $\DD = \{t \in \C \mid |t| < 1\}$ is the open disk. 
It is interesting to compare Nash's result with the setting considered in \cite{KN15}, 
where holomorphic maps from the closed disk $\ov\DD = \{t \in \C \mid |t| \le 1\}$ are studied instead. 
In that paper, Koll\'ar and N\'emethi look at the space of \emph{short arcs}, 
which are those holomorphic maps $\f \colon \ov\DD \to X$ such that
$\Supp\f^{-1}(X_\sing) = \{0\}$. The space of short arcs of a normal surface
singularity relates to the link of the singularity, and it satisfies a McKay correspondence 
property for isolated quotient singularities in all dimensions. 
In general, the space of short arcs can have 
infinitely many connected components, thus presenting a quite different behavior
from the case of formal arcs. 
\end{remark}

\begin{definition}
\label{d:fat-thin}
An irreducible set $C \subset X_\infty$ is said to be \emph{thin}
if there exists a proper closed subscheme $Z \subsetneq X$ such that $C \subset Z_\infty$. 
An irreducible set $C \subset X_\infty$ that is not thin is said to be \emph{fat}.
\end{definition}

\begin{corollary}
Every irreducible component of $\p_X^{-1}(X_\sing)$ is fat in $X_\infty$.
\end{corollary}

\begin{proof}
It suffices to observe that the arc corresponding to the
generic point of each $C_X(E_i)$ dominates the generic point of $X$. 
\end{proof}

Actually, the two cited propositions in \cite{Nas95} deal with arbitrary algebraic sets $W \subset X$;
here we are only considering the case $W = X_\sing$.  
It is asserted in \cite[Proposition~2]{Nas95}
that for an arbitrary algebraic set $W \subset X$, every irreducible component of $\p_X^{-1}(W)$
corresponds to some component of the inverse image of $W$ in the resolution.
This property does not seem to hold in such generality, at least in the way we have interpreted its meaning. 
An example where this property fails is given next.

\begin{example}
\label{eg:thin-component}
Let $X = (xy^2 = z^2) \subset \A^3$ be the Withney umbrella, as in Example~\ref{eg:1}. 
Denote by $S = (y=z=0)$ the singular locus of $X$, and let $O \in \A^3$
be the origin in the coordinates $(x,y,z)$. 

We claim that $\p_S^{-1}(O)$ is an irreducible component of $\p_X^{-1}(O)$. 
Note that $\p_S^{-1}(O)$ is irreducible since $S$ is smooth, and it is thin in $X_\infty$ since it is 
contained in $S_\infty$. In particular, it is not of the form
$C_X(E)$ for any prime divisor $E$ over $X$. In fact, as it is explained in Example~\ref{eg:1}, 
$\p_S^{-1}(O)$ is not dominated by any set in the space of arcs of any resolution of $X$. 

For short, let $F$ and $G$ respectively denote the fibers of 
$(\A^3)_3 \to \A^3$ and $X_3 \to X$ over $O$. 
Using the coordinates induced by $x,y,z$ via Hasse--Schmidt 
derivation, we have $F = \Spec k[x',y',z',x'',y'',z'',x''',y''',z''']$, and 
$G$ is defined in $F$ by the equations
$(z')^2 = 0$ and $x'(y')^2 - 2 z'z'' = 0$. In particular, $G$ has a decomposition 
into irreducible components $G_\red = V(x',z') \cup V(y',z')$. 
The image of $\p_S^{-1}(O)$ in $X_3$ is not contained in the component $V(x',z')$ of $G$
since, for instance, it contains the arc $(t,0,0)$.  
On the other hand, every arc 
$\a(t) = (x(t),y(t),z(t))$ in $\p_X^{-1}(O) \setminus \p_S^{-1}(O)$
lies over $V(x',z')$. Indeed, since $\a$ satisfies
$\a(0) = 0$ and $(y(t),z(t)) \ne (0,0)$,  
the condition $x(t)y(t)^2 = z(t)^2$ implies that the coefficients of $t$ in $x(t)$
and $z(t)$ must be zero. 
Therefore $\p_S^{-1}(O)$ is not contained in the closure of $\p_X^{-1}(O) \setminus \p_S^{-1}(O)$.
This proves our claim. 
\end{example}

Going back to the discussion leading to Theorem~\ref{t:Nash}, 
one should remark that while the index set $I$ 
depends on the choice of resolution, the irreducible decomposition
of $\p_X^{-1}(X_\sing)$ is intrinsic to $X$. The point is that $J$ may be strictly smaller than $I$, 
which means that there may be inclusions $C_i \subset C_j$. 

Suppose, for instance, that $f$ is an isomorphism over the smooth locus of $X$. 
Before taking closures, $f_\infty(\p_Y^{-1}(E_i))$ cannot be a subset of $f_\infty(\p_Y^{-1}(E_j))$
for $i \ne j$, since $f_\infty$ induces a bijection 
\[
Y_\infty \setminus (f^{-1}(X_\sing))_\infty \xrightarrow{1-1} X_\infty \setminus (X_\sing)_\infty.
\]
Away from $(f^{-1}(X_\sing))_\infty$ and $(X_\sing)_\infty$, which we can consider as subsets 
of \emph{measure zero} or \emph{infinite codimension},\footnote{These
notions can be made precise. Measure zero is intended from the point of
view of motivic integration. The codimension of a closed
subset of the space of arcs can be defined in two ways, either as the minimal dimension
of the local rings at the minimal primes, or as the limit of the codimensions
of the projections of the set to the sets of liftable jets. These two notions of codimension
may differ, but the property of being finite is equivalent in the two notions.}
$f_\infty$ is a continuous bijection but not a homeomorphism, and we can regard
the two arc spaces as being identified as sets (away from these sets of measure zero), 
with the left hand side equipped with a stronger topology. 
This explains why some sets $f_\infty(\p_Y^{-1}(E_i))$ may lie in the closure
of some other sets $f_\infty(\p_Y^{-1}(E_j))$.

One would like to be able to recognize $J$ in $I$ by only looking at 
resolution of singularities. Put another way: 
\begin{quote}
Is there a characterization of the irreducible components of $\p_X^{-1}(X_\sing)_\red$ in terms
of resolutions of $X$?
\end{quote}

This question has a natural formulation in the language of valuations. 
It is elementary to see that the set $C_X(E)$ only depends on the valuation $\val_E$ and
\[
\p_X(C_X(E)) = c_X(E),
\]
the center of $\val_E$ in $X$.\footnote{A more natural notation 
for these sets is $C_X(\val_E)$ and $c_X(\val_E)$.}  
The generic point $\a$ of $C_X(E)$ is the image of the generic point
$\~\a$ of $\p_Y^{-1}(E)$, and therefore we have
\[
\val_{\a}(h) = \val_{\~\a}(h \o f) = \ord_{E}(h\o f)
\]
for any rational function $h \in k(X)^*$.
This implies that the valuation associated to $C_X(E)$ 
is equal to the divisorial valuation defined by $E$.

\begin{remark}
The maximal divisorial set associated to a divisorial valuation $\val_E$
captures more information than just the valuation itself. Its codimension computed
on the level of jet schemes relates to the order of vanishing along $E$ of the Jacobian
of a resolution $Y \to X$, a property which follows from the results of \cite{DL99}
and is implicit in the change-of-variable formula in motivic integration. 
Because of this, maximal divisorial sets provide the essential
link between arc spaces and singularities in birational geometry.
These sets have been studied from this point of view in \cite{ELM04,Ish08,dFEI08}.
The connection between the dimension of the local ring
of $X_\infty$ at the generic point of a maximal divisorial set (or of its completion)
and other invariants of singularities that are measured by the valuation is more obscure. 
\end{remark}

Theorem~\ref{t:Nash} yields a natural identification 
between the irreducible components of $\p_X^{-1}(X_\sing)$ and certain divisorial valuations on $X$. 
Bearing this in mind, we give the following definition. 

\begin{definition}
A \emph{Nash valuation} of $X$ is the divisorial valuation $\val_C$
associated to an irreducible component $C$ of $\p_X^{-1}(X_\sing)$.
\end{definition}

One of the motivations of \cite{Nas95}
is to understand Nash valuations from the point of view of resolution of singularities.
A precise formulation of this problem, which is discussed in Section~\ref{s:higher-dim}, 
has become known as the \emph{Nash problem}.

To address this problem, given an arbitrary resolution $f \colon Y \to X$
and a prime divisor $E$ contained in $f^{-1}(X_\sing)$, 
one needs to analyze the condition that $C_X(E)$
is strictly contained in some irreducible component of $\p_X^{-1}(X_\sing)$.

The idea, which goes back to Lejeune-Jalabert \cite{LJ80}, 
is to detect such proximity by producing a 1-parameter family
of arcs which originates in $C_X(E)$ and moves outside of it but still
within $\p_X^{-1}(X_\sing)$. As intuitive as it may be, the
existence of such family of arcs is a delicate fact. 
 
The following \emph{curve selection lemma} formalizes this idea.
It should be clear that the arc $\Phi$ on $X_\infty$ provided by the theorem
gives the desired family of arcs on $X$.
This result is the key technical tool needed to address the Nash problem.

\begin{theorem}[Reguera \protect{\cite[Corollary~4.8]{Reg06}}]
\label{t:CSL}
Suppose that $C_X(E) \subsetneq C_X(F)$ 
for some prime divisors $E$ and $F$ over $X$.
Then there is an arc
\[
\Phi \colon \Spec K[[s]] \to X_\infty
\]
such that $\Phi(0) \in C_X(E)$ is the generic point of $C_X(E)$ and 
$\Phi(\e) \in C_X(F) \setminus C_X(E)$.\footnote{We are being somewhat sloppy here. 
To be precise, $\Phi(0)$ dominates the generic point of $C_X(E)$ and 
$\Phi(\e)$ dominates a point in $C_X(F) \setminus C_X(E)$. The field $K$ can be chosen 
to be a finite extension of the residue field of the generic point of $C_X(E)$.}
\end{theorem}

\begin{remark}
The theorem is stated more generally for a larger class of subsets of $X_\infty$
called \emph{generically stable sets} (we refer to \cite{Reg06} for the precise definition).
The fact that $C_X(E)$ and $C_X(F)$ are generically stable sets
is a consequence of Theorem~\ref{t:f-infty-constr}.
\end{remark}

In the Noetherian setting, the curve selection lemma essentially follows by cutting 
down to a curve, normalizing, and completing.
The curve selection lemma however fails in general for non Noetherian schemes.

\begin{example}
Let $C = \Spec k[x_1,x_2,x_3,\dots]/I$ where $I$ is the ideal generated by the
polynomials $x_1-(x_i)^i$ for $i \ge 2$. 
Let $\Phi \colon \Spec k[[s]] \to C$ be any morphism such that $\Phi(0) = O$, the origin
of $C$. Writing $\Phi(s) = (x_1(s),x_2(s),x_3(s),\dots)$, we have $\ord_sx_i(s) \ge 1$ for all $i$. 
From the equations $x_1(s) = x_i(s)^i$, we deduce that $x_i(s) = 0$ for every $i$, and hence
$\Phi$ is the constant arc.  
We note that $C$ can be realized as a closed irreducible set in the space
of arcs of any variety. 
\end{example}

The main result behind the curve selection lemma is another theorem of Reguera
stating that if $\a$ is the generic point of $C_X(E)$ (or, more generally,
if $\a \in X_\infty$ is what is called a \emph{stable point}), then 
the completed local ring $\widehat{\O_{X_\infty,\a}}$ is Noetherian \cite[Corollary~4.6]{Reg06}. 
Once this property is established, the proof of the curve selection lemma follows 
as a fairly standard application of Cohen's structure theorem 
(see \cite{Reg06} for details).

\begin{remark}
The fact that $\widehat{\O_{X_\infty,\a}}$ is a Noetherian ring is a delicate property. 
There are examples where, before completion, the local ring 
$\O_{X_\infty,\a}$ is not Noetherian \cite[Example~3.16]{Reg09}.
\end{remark}

\begin{remark}
A related result of Grinberg and Kazhdan \cite{GK00}, reproved and extended to all characteristics
by Drinfeld \cite{Dri}, states that if $\g \in X_\infty \setminus (X_\sing)_\infty$ is
a $k$-valued point, then there is an isomorphism
\[
\widehat{\O_{X_\infty,\g}} \cong k[[x_1,x_2,x_3,\dots]]/I
\]
where $I$ is the extension of an ideal in a finite dimensional polynomial ring 
$k[x_1,\dots,x_n]$.
We will not use this result in this paper. 
\end{remark}

\section{Dimension one}

The arc space of a curve is fairly easy to understand. 
Let $X$ be a curve over an algebraically closed field of characteristic zero, 
and suppose that $P \in X$ is a singular point. 
Let $f \colon Y \to X$ be the normalization, and write 
$f^{-1}(P) = \{Q_1,\dots,Q_r\}$. Note that $r$ is the number of analytic branches of $X$ at $P$. 

\begin{proposition}
The fiber $\p_X^{-1}(P)$ has $r$ irreducible components.
For every $i=1,\dots,r$, the set $f_\infty(\p_Y^{-1}(Q_i)_\red)$ is closed and is
one of the irreducible components of $\p_X^{-1}(P)$.
\end{proposition}

\begin{proof}
For any field extension $K/k$, every constant $K$-valued arc in 
$X$ through $P$ has $r$ distinct lifts to $Y$, 
each mapping to a distinct point $Q_i$. By contrast,
every non-constant $K$-valued arc through $P$
lifts uniquely to an arc on $Y$ which passes through one of the $Q_i$. 
This shows two things: for every $i$, the image 
$f_\infty(\p_Y^{-1}(Q_i)_\red)$ is equal to $C_X(Q_i)$ and hence is closed,
and, for every $i \ne j$, the intersection of 
$f_\infty(\p_Y^{-1}(Q_i)_\red) \cap f_\infty(\p_Y^{-1}(Q_j)_\red)$
consists only of the trivial arc at $P$. The proposition follows from these two properties. 
\end{proof}

\begin{remark}
The sets of jets through the singularity of a curve may of course have more irreducible components
than the number of branches. 
The case of a node $X = (xy=0) \subset \A^2$ provides an elementary example: for every $m \ge 1$, 
the fiber over the origin $O \in X$ of the truncation map $\t_m \colon X_m \to X$ has 
a decomposition into $m$ irreducible components
\[
\t_m^{-1}(O)_\red = \bigcup_{i+j=m+1} C_{i,j}
\]
where $C_{i,j}$ is the closure of the set of
$m$-jets on $\A^2$ with order of contact $i \ge 1$ along $(x=0)$
and $j \ge 1$ along $(y=0)$. 
\end{remark}

\section{Dimension two}

Throughout this section, suppose that $X$ is a surface defined
over an algebraically closed field of characteristic zero. There is a natural set
of divisorial valuations that one can regard in connection to the Nash valuations, namely, 
the set of divisorial valuations $\val_{E_i}$ associated to the
exceptional divisors $E_1,\dots,E_m$ in the minimal resolution of singularities
\[
f \colon Y \to X.
\]
Since we are not assuming that $X$ is normal, 
we should stress that a prime divisor $E$ on $Y$ is defined to be exceptional over $X$
if $f$ is not an isomorphism at the generic point of $E$.

In his paper, Nash asked whether there exists a natural one-to-one correspondence 
between the irreducible components of $\p_X^{-1}(X_\sing)$ and the exceptional divisors
in the minimal resolution of $X$ (that is, the irreducible components of $f^{-1}(X_\sing)$), 
the correspondence given indeed by identification between the associated valuations. 
Nash verified the question for ${\bf A}_n$ singularities, where the correspondence is
not hard to check.

A particularly simple case 
which already illustrates in concrete terms the geometry of the correspondence is that of 
an ${\bf A}_2$ singularity.

\begin{example}
\label{eg:A2-sing}
Let $X = (xy = z^3) \subset \A^3$. The blow-up of the origin $O \in X$ gives
the minimal resolution $f \colon Y \to X$. Let $U \subset Y$ be the affine chart 
with coordinates $(u,v)$ where $f$ is given by $(x,y,z) = (u^2v,uv^2,uv)$.
The two exceptional divisors $E_1,E_2$ are given in $U$ by $E_1 = (u=0)$ and $E_2 = (v=0)$. 
Let $\g(t) = (x(t),y(t),z(t))$ be an arbitrary arc on $X$ through $O$. 
The power series
\[
x(t) = \sum_{i =1}^\infty a_it^i, \quad
y(t) = \sum_{i =1}^\infty b_it^i, \quad
z(t) = \sum_{i =1}^\infty c_it^i
\]
satisfy the equation $x(t)y(t) = z(t)^3$. 
Expanding, this gives
\begin{multline*}
a_1b_1t^2 + (a_1b_2 + a_2b_1)t^3 + \dots + \Big(\sum_{i+j=m} a_ib_j\Big)t^m + \dots \\
= c_1^3 t^3 + \dots + \Big(\sum_{i+j+k=m} c_ic_jc_k\Big)t^m + \dots
\end{multline*}
Comparing the coefficients of $t^2$, we get the equation $a_1b_1=0$. This leads to two cases. 

Suppose $a_1=0$. Generically, we have $b_1 \ne 0$, and hence we can solve all remaining
equations for $a_i$ ($i \ge 2$) in terms of the $b_j$ and $c_k$, which are free parameters. 
This gives an irreducible component $C_1$ of $\p_X^{-1}(O)$ whose generic arc
$\a(t) = (x(t),y(t),z(t))$ has first entry of order $\ord_t(x(t)) = 2$,
and the other two entries have order one. Write
\[
\a(t) = (t^2 \.\ov x(t), t \.\ov y(t), t \.\ov z(t))
\]
where $\ov x(t),\ov y(t),\ov z(t)$ are units. 
Using the equations $u = x/z$ and $v = y/z$, the lift $\~\a$ of $\a$ to $Y$ has entries
\[
\~\a(t) = \Big(t\.\frac{\ov x(t)}{\ov z(t)}, \frac{\ov y(t)}{\ov z(t)}\Big)
\]
in the coordinates $(u,v)$ of $U$. This shows that $\~\a(t)$ has order of contact one
with $E_1$ and order of contact zero with $E_2$. 
In fact, one can argue that $\~\a$ is the generic point of $\p_Y^{-1}(E_1)$. 

Taking $b_1=0$, we get in a similar way the other component $C_2$ of $\p_X^{-1}(O)$, 
which corresponds to $E_2$. 
\end{example}

The simplicity of this example can be misleading. While 
arc spaces of ${\bf A}_n$ singularities
are still fairly easy to understand \cite{Nas95}, 
it was only recently that an answer to Nash's question was given
for ${\bf D}_n$ singularities \cite{Ple08}
and for ${\bf E}_6,{\bf E}_7,{\bf E}_8$ \cite{PS12,PP13}.
The fact is that, even when dealing with very simple equations like those of rational double points, 
the complexity of the equations of the arc space can grow very quickly.
The case of sandwiched singularities was solved in \cite{LJR99}, 
and a general proof for all rational surface singularities was given in \cite{Reg12}. 
Some families of non-rational surface singularities where Nash's question has a positive anwer
were found in \cite{PPP06}. 

The answer to Nash's question given in \cite{PP13} for quotient surface singularities
uses the reduction to the problem to 
a topological setting due to \cite{FdB12}.
Following the same approach, a complete proof valid for all surfaces 
was finally found by Fernandez de Bobadilla and Pe Pereira. 

\begin{theorem}[Fernandez de Bobadilla and Pe Pereira \protect{\cite[Main Theorem]{FdBPP12}}]
\label{t:Nash-dim=2}
A valuation on a surface $X$ is a Nash valuation if and only if 
it is the valuation associated to an exceptional divisor on the minimal resolution of $X$. 
\end{theorem}

We present here a purely algebraic proof of this result 
that is based on the proof of the main theorem of \cite{dFD16}.

\begin{proof}[Proof of Theorem~\ref{t:Nash-dim=2}]
Let $f \colon Y \to X$ be the minimal resolution of singularities. 
Given what we already discussed about the decomposition of $\p_X^{-1}(X_\sing)$ into
irreducible components, in order to prove the theorem we only need to show that
if $E$ is a prime exceptional divisor on $Y$, then $C_X(E)$ is an irreducible 
component of $\p_X^{-1}(X_\sing)$. 

We proceed by way of contradiction
and assume that $C_X(E)$ is not an irreducible component of $\p_X^{-1}(X_\sing)$. 
This means that $C_X(E)$ is contained in $C_X(F)$ for some other exceptional divisor $F$. 

Let $p \in E$ be a very general closed point.\footnote{By \emph{very general}, 
we mean that the point is taken in the complement of countably many
proper closed subsets.}
By applying Theorem~\ref{t:CSL} in conjuction with a suitable specialization argument 
(\cite[Proposition~2.9]{LJR12}, \cite[Theorem~7.6]{dFD16}), we obtain an arc
\[
\Phi \colon \Spec k[[s]] \to X_\infty
\]
on the space of arcs of $X$ such that
\begin{enumerate}
\item
$\a_0 := \Phi(0)$ is a $k$-valued arc on $X$ whose lift $\~\a_0$ to $Y$
is an arc with order of contact one with $E$ at $p$ (i.e., $\~\a_0(0) = p$
and $\ord_{\~\a_0}(E) = 1$),
\item
$\a_\e := \Phi(\e)$ is a $k(\hskip-1pt(s)\hskip-1pt)$-valued point of $\p_X^{-1}(X_\sing) \setminus C_X(E)$. 
\end{enumerate}

By definition, $\Phi$ is a formal 1-parameter family of arcs
giving an infinitesimal deformation of $\a_0$ in $X_\infty$. 
We think of $\Phi$ as a morphism
\[
\Phi \colon S = \Spec k[[s,t]] \to X, \quad \Phi(s,t) = \a_s(t)
\]
from a 2-dimensional regular germ to $X$.
Conditions~(a) and~(b) imply that the rational map 
\[
\~\Phi := f^{-1} \o \Phi \colon S \rat Y
\]
is not well-defined. Let $g \colon Z \to S$ be the minimal sequence of
monomial transformations resolving the indeterminacies of $\~\Phi$, 
and let $g' \colon Z' \to S$ be the normalized blow-up of the
ideal $\Phi^{-1}\fa\.\O_S$ where $\fa \subset \O_X$ is an ideal such that $Y = \Bl_\fa X$.
We have the commutative diagram
\[
\xymatrix@C=8pt{
G \ar@{}[r]|-\subset 
& Z \ar@/_1pc/[rrdd]_{g} \ar[rrd]^{h}\ar@/^1pc/[rrrrd]^\f &&&&&\\
&&& Z' \ar[d]_{g'}\ar[rr]^{\f'} && Y \ar[d]^f \ar@{}[r]|-\supset & E \\
&&& S \ar[rr]^\Phi \ar@{-->}[rru]^{\~\Phi} && X 
}
\]
where we denote by $G$ the $g$-exceptional divisor intersecting
the proper transform $T$ of the $t$-axis $(s=0) \subset S$. 
The morphisms $\f$ and $\f'$ are induced by resolving
the indeterminacies of $\~\Phi$, and $h$ is the morphism 
contracting all $g$-exceptional curves that are contracted to a point by $\f$. 
One can check that $Z$ is regular, $Z'$ has rational singularities and hence is $\Q$-factorial,
and $h$ is the minimal resolution of singularities of $Z'$
(see \cite[Proposition~4.1]{dFD16} for details).

The image of the exceptional locus $\Ex(g)$ of $g$ in $Y$
is contained in the exceptional locus of $f$, and so is 
the image of the exceptional locus $\Ex(g')$ of $g'$.
Recall that none of the irreducible components of $\Ex(g')$ is contracted by $\f'$. 
Since $p$ was picked to be a general point of $E$, and it belongs to $\f(G)$, 
every irreducible component of $\Ex(g')$ 
that contains $h(G)$ must pass through $p$ and hence dominate $E$. 
Note that there is at least one such component of $\Ex(g')$, 
since $g'$ is not an isomorphism. 
This implies that $\f(Z)$ contains the generic point of $E$. On the other hand, 
$\f(Z)$ is not contained in $E$, since $\f(T)$ is an arc
on $Y$ with finite order of contact with $E$ and hence not entirely contained in $E$. 
We conclude that $\f$ is a dominant map.

Let $K_{Z/Y}$ be the relative canonical divisor of $Z$ over $Y$, 
locally defined by the Jacobian ideal $\Jac_\f \subset \O_Z$, and let
$K_{Z'/Y} = h_*K_{Z/Y}$, which we think of as the relative canonical divisor of $Z'$ 
over $Y$. Similarly, let $K_{Z/S}$ be the relative canonical divisor of $g$ and 
let $K_{Z'/S} = h_*K_{Z/S}$.\footnote{The formal 
definition of relative canonical divisor via sheaves of differentials
is not straigtforward, as $\Om_{Z/k}$ is not the right object in this setting.
For a correct formal definition, one needs to replace $\Om_{Z/k}$
with the sheaf of \emph{special differentials}. Once this adjustment is done, 
the same definitions and properties follow into place as in the usual setting, 
and therefore we shall omit this discussion here. For details, 
we refer to \cite[Section~4]{dFD16}.}
We decompose 
\[
K_{Z'/Y} = K_{Z'/Y}^{\text{$g'$-exc}} + K_{Z'/Y}^{\text{$g'$-hor}}
\]
where every component of $K_{Z'/Y}^{\text{$g'$-exc}}$ is $g'$-exceptional
and none of the components of $K_{Z'/Y}^{\text{$g'$-hor}}$ is. 

We claim that the following series of inequalities hold:
\[
1 \overset{(1)}{\le} \ord_G(K_{Z/S}) \overset{(2)}{\le} \ord_G(h^*K_{Z'/S}) 
\overset{(3)}{\le} \ord_G(h^*K_{Z'/Y}^{\text{$g'$-exc}}) \overset{(4)}{<} \ord_G(\f^*E) \overset{(5)}{=} 1.
\]
This clearly gives a contradiction, which is what we are after. 

The reminder of the proof is devoted to explain these inequalities. We proceed with one
inequality at a time.

\medskip
\noindent
{\it Inequality} (1). 
The fact that $\ord_G(K_{Z/S}) \ge 1$ holds simply because
$S$ is regular and $G$ is $g$-exceptional. 

\medskip
\noindent
{\it Inequality} (2). 
This inequality follows from the fact that the $\Q$-divisor
\[
K_{Z/Z'} = K_{Z/S} - h^*K_{Z'/S}
\]
is $h$-nef and $h$-exceptional since $h$ is the minimal resolution of singularities
of $Z'$, and hence is anti-effective by the negative definiteness 
of the intersection matrix of the $h$-exceptional divisors (see \cite[Proposition~4.12]{dFD16}). 
Here we are using the fact that $Z'$, having rational singularities, is $\Q$-factorial
and therefore the pull-back $h^*K_{Z'/S}$ is defined.

\medskip
\noindent
{\it Inequality} (3). 
Here is where we use the fact that $f$ is the minimal resolution of singularities of $X$.
First, notice that the divisor 
\[
K_{Z'/S} - K_{Z'/Y}^{\text{$g'$-exc}}
\]
is $g'$-exceptional. We claim that this divisor is also $g'$-nef. Indeed, we have 
\[
K_{Z'/S} - K_{Z'/Y}^{\text{$g'$-exc}} \sim K_{Z'} - K_{Z'/Y} + K_{Z'/Y}^{\text{$g'$-hor}}
\sim (\f')^*K_Y + K_{Z'/Y}^{\text{$g'$-hor}}. 
\]
Since $f$ is the minimal resolution, $K_Y$ is $f$-nef, and hence 
$(\f')^*K_Y$ is $g'$-nef. On the other hand, $K_{Z'/Y}^{\text{$g'$-hor}}$ is clearly $g'$-nef
because it is effective and contains no $g'$-exceptional divisors. 
Therefore $K_{Z'/S} - K_{Z'/Y}^{\text{$g'$-exc}}$ is $g'$-nef, as claimed. 
We conclude that this divisor is anti-effective, 
and this gives the third inequality.

\medskip
\noindent
{\it Inequality} (4). 
Let $C_1,\dots,C_n$ be the irreducible components of $\Ex(g')$ containing $h(G)$.
Each $C_i$ dominates $E$, and we have $\ord_{C_i}(K_{Z'/Y}) = \ord_{C_i}((\f')^*E) - 1$
by a Hurwitz-type computation. This implies that
\[
\ord_G(h^*K_{Z'/Y}^{\text{$g'$-exc}}) < \ord_G(\f^*E)
\]
(see \cite[(5.4)]{dFD16} for more details).
Here we are using again that $Z'$ is $\Q$-factorial. 

\medskip
\noindent
{\it Equality} (5). 
This follows by the way we chose $\Phi$. Recall that $\a_0 = \Phi(0)$ lifts to an arc 
$\~\a_0 \colon \Spec k[[t]] \to Y$
with order of contact one along $E$. This arc is parametrized by the $t$-axis of $S$. 
This means that $\~\a_0$ factors through a morphism $\ff \colon \Spec k[[t]] \to Z$
which gives a parameterization of the proper transform $T$ of the $t$-axis.
Since
\[
1 = \ord_{\~\a_0}(E) = \ord_t(\~\a_0^*E) \ge \ord_t(\ff^*G) \. \ord_G(\f^*E), 
\]
we conclude that $\ord_G(\f^*E)=1$ (see the discussion leading to \cite[(5.5)]{dFD16}). 
This proves~(5) and hence completes the proof of the theorem.
\end{proof}

We conclude this section with a brief discussion of the original
proof of Theorem~\ref{t:Nash-dim=2}, referring the reader to the original papers \cite{FdB12,FdBPP12}
and the survey \cite{PS15} for more rigorous and detailed proofs. 

The first step is to reduce to the case where $k = \C$ and $X$ is normal. 
Once in this situation, let $f \colon Y \to X$ be the minimal resolution. 
For simplicity, we assume that the exceptional locus of $f$ is a divisor 
with simple normal crossings. The proof of the general case is similar but it requires
an argument on local deformation to the Milnor fiber which we prefer to omit here.

As usual, one assumes by contradiction that there are two exceptional divisors $E$ and $F$
on $Y$ such that $C_X(E) \subset C_X(F)$. 
Like in the algebraic proof we gave above, the curve selection lemma
yields a map $\Phi \colon S = \Spec \C[[s,t]] \to X$ with the properties listed in the proof. 
Such a map is called a \emph{formal wedge}. By the results of \cite{FdB12}
which rely on Popescu's approximation theorem, 
one can replace $\Phi$ with a \emph{convergent wedge}, and hence assume without
loss of generality that $S \subset \C^2$ is a small open neighborhood of 
the origin.\footnote{A rigorous discussion of what follows 
requires working with Milnor representatives of $X$ and the wedge.}   

Fix a sufficiently small $\ep > 0$, and let $\DD_\ep = \{ t \in \C \mid |t| < \ep \}$.
For every $s \in \C$ with $|s| < \ep$, we have a holomorphic map 
\[
\a_s \colon \DD_\ep \to X, \quad \a_s(t) := \Phi(s,t).
\]
The image of this map is not contained in $X_\reg$ and lifts uniquely
to a holomorphic map $\~\a_s \colon \DD_\ep \to Y$. Let $D_s \subset Y$ denote the image
of $\~\a_s$. Since $D_0$ is the support of a small curve whose germ at the point of contact
with $E$ is a smooth arc, it is homeomorphic to an open disk. 
One deduces from this that if $s$ is sufficietly small then  $D_s$ is homeomorphic 
to an open disk.

As $s$ approaches $0$, $D_s$ degenerates to a cycle 
\[
D_0 + \sum a_i E_i
\]
supported within the union of $D_0$ and the exceptional divisor $\Ex(f) = \sum E_i$.
Let $\G = \Supp (D_0 + \sum a_i E_i)$. Let $I$ be the union of $\{0\}$
with the index set of the components $E_i$
appearing in $\G$, and let $J$ be the index set for the singular points $p_j$ of $\G$. 
Note that $E = E_i$ for some $i \in I$, say for $i = 1$. 
Suppose that $0 \in J$ is the index such that $p_0$ is the point of intersection 
of $D_0$ with $E_1$. 

For every $j \in J$, let $B_j \subset Y$ be a small ball around $p_j$, and for every
$i \in I$, let $T_i$ be a small tubolar neighborhood of $D_0$ if $i = 0$, and of $E_i$ if $i \ne 0$. 
We assume that the sectional radius of $T_i$ is chosen sufficiently small with respect to 
the radius of the balls $B_j$ 
so that the boundary of $T_i$ intersects the boundary of $B_j$ transversally and all such 
intersections are disjoint. Let $T^\o_i$, $D_0^\o$, and $E_i^\o$ denote the restrictions of
$T_i$, $D_0$, and $E_i$ to the complement of $\bigcup B_j$. 
Fix $s$ with $0 < |s| \ll 1$ so that 
$D_s$ is contained in $(\bigcup T_i) \cup (\bigcup B_j)$. We assume that $D_s$ intersects transversally 
the boundary of each $B_j$. We have
\[
\chi(D_s) = \sum \chi(D_s \cap T^\o_i) + \sum \chi(D_s \cap B_j),
\]
where $\chi$ is the Euler--Poincar\'e characteristic.

For $i = 0$, we have $\chi(D_s \cap T^\o_0) = \chi(D_0^\o) = 0$, and for $i \ne 0$
we have
\[
\chi(D_s \cap T^\o_i) \le a_i\,\chi(E^\o_i) 
\]
by Hurwitz formula.\footnote{Here we
are implicitly using that the boundaries of $D_s \cap T^\o_i$ and $D_i^\o$ are 
unions of circles, and hence they can be added in without altering the computation.}
To bound $\chi(D_s \cap B_j)$, we observe that $D_s \cap B_j$ is a union of
disjoint orientable surfaces with boundary. Those homeomorphic to the disk are the
only components contributing positively to the characteristic, and each such component
must intersect $\G$ at some point in $B_j$. It follows that
\[
\chi(D_s \cap B_j) \le \sum_{p \in B_j} i_p(D_s,\G),
\]
where we denote by $i_p$ the intersection multiplicity at a point $p$
(see \cite[Lemma~7]{FdBPP12} for more details). 
For $j = 0$, this estimate can be improved. Indeed, $D_s \cap B_0$
must have at least one connected component whose boundary is the union of at least two
circles, one contained in $T_0 \cap B_0$ and the other contained in $T_1 \cap B_0$. 
Such component intersects both branches of $\G \cap B_j$ and does not contribute
positively to the characteristic. 
This implies that 
\[
\chi(D_s \cap B_0) \le -2 + \sum_{p \in B_0} i_p(D_s,\G)
\]
(we refer to the discussion leading to \cite[(12)]{FdBPP12} for more details).
Putting everything together and suitably rearranging the terms, one gets
\[
\chi(D_s) = \sum \chi(D_s \cap T^\o_i) + \sum \chi(D_s \cap B_j) \le \sum a_i(2-2g(E_i) + E_i^2).
\]
By the adjunction formula, the right-hand side is equal to $-K_Y\.\sum a_i E_i$. 
As $K_Y$ is nef over $X$ ($f$ being the minimal resolution)
and $\sum a_i E_i$ is $f$-exceptional, this number is $\le 0$.
Since, on the other hand, $D_s$ is homeomorphic to the unit disk and hence $\chi(D_s) = 1$, 
we get a contradiction.

\section{Higher dimensions}
\label{s:higher-dim}

Moving on to higher dimensional singularities, 
it becomes less clear which exceptional divisors should correspond to Nash valuations.
The reason is that in dimension $\ge 3$ there is no minimal resolution available to 
determine a natural set of candidates. In fact, some varieties
may have small resolutions, which extract no divisors at all. 
With this in mind, Nash proposed to consider the following set of valuations. 

Throughout this section, let $X$ be a variety of positive dimension defined over 
an algebraically closed field of characteristic zero.

\begin{definition}
\label{d:essential}
An \emph{essential valuation} of $X$ is a divisorial valuation whose center on every resolution
of singularities $f \colon Y \to X$ is an irreducible component of $f^{-1}(X_\sing)$.
\end{definition}

\begin{proposition}[Nash \protect{\cite[Corollary]{Nas95}}]
\label{p:Nash-implies-essential}
Every Nash valuation of $X$ is essential.
\end{proposition}

\begin{proof}
Let $\val_C$ be the Nash valuation associated to an irreducible component of $\p_X^{-1}(X_\sing)$. 
We already know that $\val_C$ is a divisorial valuation. 
Let $f \colon Y \to X$ be an arbitrary resolution of singularities.
As we argued in the proof of Theorem~\ref{t:Nash}, there is an irreducible component 
$E$ of $f^{-1}(X_\sing)$ such that 
\[
C = \ov{f_\infty(\p_Y^{-1}(E))}.
\]
In other words, the generic point of $C$ is the image of the generic point of $\p_Y^{-1}(E)$.
This implies that the center of $\val_C$ in $Y$ is $E$.
Since $f$ is an arbitrary resolution, we conclude that $\val_C$ is an essential valuation.
\end{proof}

\begin{definition}
After identifying the irreducible components of $\p_X^{-1}(X_\sing)$ with the valuations they define, 
the inclusion of the set of Nash valuations into the set of essential valuations
is known as the \emph{Nash map}.
\end{definition}

Nash asked whether the Nash map is surjective, that is, whether the property of being essential
characterizes Nash valuations. This question became known as the \emph{Nash problem}. 

Theorem~\ref{t:Nash-dim=2} states that this is the case 
in dimension two. However, after years of speculation, this turned out to be false in general:
counter-examples where first found in dimensions $\ge 4$ \cite{IK03}, 
and later in dimension 3 as well \cite{dF13}. 
A larger class of counter-examples
showing that this phenomenon is actually quite common and not limited to few sporadic examples
was finally produced in \cite{JK13}. 

In \cite{JK13}, Nash valuations of a ${\bf cA}$-type singularity
$X = (xy = f(z_1,\dots,z_n)) \subset \A^{n+2}$ (where $\mult(f) \ge 2$)
are completely determined, and essential valuations are characterized 
when $\mult(f) = 2$. A special case of their result, stated next, shows that 
the Nash map is not surjective about half of the times for 3-dimensional ${\bf cA}_1$ singularities.

\begin{theorem}[Johnson and Koll\'ar \protect{\cite[Theorem~1 and Proposition~9]{JK13}}]
\label{t:JK}
For $m \ge 3$, the singular threefold $X = (xy = z^2-w^m) \subset \A^4$ has one Nash valuation,
and the number of essential valuations of $X$ is one if $m$ is even or $m=3$, 
and two if $m$ is odd $\ge 5$.
\end{theorem}

We extract from this result the case $m=5$, which gives the simplest counter-example
to the Nash problem. We review the proof of this case. The proof of Lemma~\ref{l:1}, 
which gives the count of Nash valuations for this example, formalizes the type of discussion
given in Example~\ref{eg:A2-sing} based on localization and elimination of variables, 
and is inspired by some computations we learned from Ana Reguera.

\begin{corollary}
\label{c:JK-m=5}
The Nash map is not surjective for $X = (xy=z^2-w^5) \subset \A^4$.
\end{corollary}

\begin{proof}
A resolution of $X$ can be obtained by taking two blow-ups. 
The blow-up $f \colon Y \to X$ of the origin $O$ produces a model
with an isolated singularity $P \in Y$ whose tangent cone is 
the affine cone over a (singular) quadric surface. Blowing up
the point $P$ gives a resolution $g \colon Z \to Y$ of $Y$, 
and hence of $X$. Let $F \subset Y$ be the exceptional divisor of $f$
and $G \subset Z$ the exceptional divisor of $g$. 

Since $Z \to X$ is a resolution which only extracts two divisors, it follows 
that there are at most two essential valuations. Moreover, we have
\[
\p_X^{-1}(O)_\red = C_X(F) \cup C_X(G),
\]
and hence there are at most two Nash valuations. 
The precise count of Nash valuations and essential valuations
is given in the next two lemmas which, combined, yield the corollary.
\end{proof}

\begin{lemma}
\label{l:1}
With the above notation, we have $\p_X^{-1}(O)_\red = C_X(F)$, and 
therefore $\val_F$ is the only Nash valuation of $X$.
\end{lemma}

\begin{proof}
First note that $C_X(F) \not\subset C_X(G)$ because every arc $\g \in C_X(G)$
has $\ord_\g(x) \ge 2$, whereas if $\a \in C_X(F)$ is the generic point then $\ord_\a(x) = 1$.
Therefore the statement is equivalent to showing that $\p_X^{-1}(O)$ is irreducible. 
This set is defined by the vanishing of the derivations of the polynomial
\[
h(x,y,z,w) = xy - z^2 + w^5
\]
and the pull-back of the maximal ideal. Explicitly, 
let $(x_i,y_i,z_i,w_i)_{i \ge 0}$ denote the coordinates of $(\A^4)_\infty$
defined by setting 
$x_i = D_i(x)$, $y_i= D_i(y)$, $z_i = D_i(z)$, and $w_i= D_i(w)$,
where $D_i$ are the Hasse--Schmidt derivations. 
The fiber $\p_X^{-1}(O)$ is defined in $(\A^4)_\infty$ by the equations
\[
x_0 = y_0 = z_0 = w_0 = 0 \and h_i := D^i(h(x_0,y_0,z_0,w_0)) = 0 \for i \ge 0. 
\]
Let $\ov h_i$ denote the polynomial $h_i$ once we set $x_0 = y_0 = z_0 = w_0 = 0$.
Note that $\ov h_0$ and $\ov h_1$ vanish identically. 

Let $U \subset \p^{-1}_X(O)$ be the open set obtained by
inverting $w_1$ and $x_2$. 
Using the equations $\ov h_i = 0$, for $i \ge 6$, to eliminate
the variables $w_j$ for $j \ge 2$, we see that $U = \Spec R$ where
\[
R = \big(k[w_1][x_i,y_i,z_i]_{i \ge 1}/(\ov h_2,\dots,\ov h_5)\big)_{w_1x_2}.
\]
We claim that $R$ is a domain. 
Since the polynomials $\ov h_2,\dots,\ov h_5$ do not depend on the variables $x_i,y_i,z_i$
for $i \ge 5$, it suffices to show that 
\[
S = \big(k[w_1][x_i,y_i,z_i]_{1 \le i \le 4}/(\ov h_2,\dots,\ov h_5)\big)_{w_1x_2}
\]
is a domain. It can be checked that $S_{x_1}$ is a domain of dimension 9 and $S/(x_1)$ 
is a domain of dimension 8. From this it follows that $S$, and hence $R$, are domains, 
and therefore $U$ is irreducible. 
Since $U$ has nonempty intersection
with both $C_X(F)$ and $C_X(G)$, we conclude that $\p_X^{-1}(O)$ is irreducible. 
\end{proof}

\begin{lemma}
With the above notation, both $\val_F$ and $\val_G$ are essential valuations of $X$, 
and therefore $X$ has two essential valuations.
\end{lemma}

\begin{proof}
The fact that $\val_F$ is essential follows by Lemma~\ref{l:1} and 
Proposition~\ref{p:Nash-implies-essential}. Suppose that $\val_G$ is not essential. Then there is
a resolution $\m \colon W \to X$ such that the center $C = c_W(G)$ is not an irreducible 
component of $\m^{-1}(O)$.
Let $T$ be an irreducible component of $\m^{-1}(O)$ containing $C$. 
We have the commutative diagram
\[
\xymatrix@C=8pt{
& G \ar@{}[r]|\subset & Z \ar[d]_g \ar@{-->}[drr]^\ff &&&& \\
P \ar@{}[r]|\in & F \ar@{}[r]|\subset & Y \ar[d]_f && 
W \ar@{-->}[ll]_\f \ar[dll]^\m & T \ar@{}[l]|\supset & C \ar@{}[l]|\supset \\
& O \ar@{}[r]|\in & X &&&&
}
\]

For short, for any prime divisor $D$ over a $\Q$-Gorenstein variety $V$
we denote by $k_D(V)$ the coefficient 
of $D$ in the relative canonical divisor $K_{V'/V}$ on some (smooth or normal) model $V'$ over $V$
on which $D$ is a divisor, 
and call it the \emph{discrepancy} of $D$ over $V$.\footnote{Here we are a bit sloppy
and identify divisors across different models when they define the same valuation.}

A direct computation shows that $k_F(X)= 1$ and $k_G(X) = 2$. In particular 
$X$ is terminal and $F$ is the only exceptional divisor over $O\in X$ with discrepancy one.

Since $X$ is factorial, the exceptional locus $\Ex(\m)$ of $\m$ is a divisor, 
as otherwise the push-forward 
of a general hyperplane section of $W$ would be a Weil divisor in $X$ that is not Cartier.
Note that $K_{W/X} \ge \Ex(\m)$ because $X$ has terminal singularities. 
This implies that $k_G(W) \le k_G(X) - 1 = 1$. 
Since, on the other hand, $k_G(W) \ge \codim_W(C) - 1 \ge 1$ because $W$ is smooth, we conclude that
$k_G(W) = 1$, $C$ is a 1-dimensional set contained in a unique $\m$-exceptional divisor $E$, and 
$k_E(X) = 1$. 

Since $C \subsetneq T \subset E$, 
we must have $T = E$, and hence $E$ is a divisor with center $O$ in $X$.
Since there is only one exceptional divisor over $X$ with discrepancy one and center $O$, 
we deduce that $E$ is the proper transform of $F$.
Taking into account that $\val_E(\fm_{X,O}) = \val_G(\fm_{X,O}) = 1$, 
we see that $\fm_{X,O} \.\O_W$ is equal to $\O_W(-E)$ 
in a neighborhood of the generic point of $C$, 
and hence it is locally principal there. 
As $f$ is the blow-up of $\fm_{X,O}$,
this means that the map $\f \colon W \rat Y$ is well-defined at the generic point of $C$. 
Replacing $W$ with a higher model without blowing-up near the generic point of $C$, 
we can assume that $\f$ is everywhere well defined and projective.
Note that $\f$ contracts $C$ to the point $P$, since $g(G) = P$.
Then $\f$ is a resolution of $Y$ whose exceptional locus has a 
component of codimension 2.
This contradicts the fact that $Y$ is factorial.
\end{proof}

In the surface case, it is clear that a divisorial valuation is essential if and only
if it is defined by an exceptional divisor on the minimal resolution. 
In the terminology introduced in this section,
Theorem~\ref{t:Nash-dim=2} simply states that the Nash map is surjective in dimension two. 

Even though the Nash map is not always surjective in higher dimensions, 
the result on surfaces still admits a natural generalization to all dimensions. 
This is possible by interpreting the minimal resolution of a surface from 
the point of view of the minimal model program. 
With this in mind, we give the following definition. 

\begin{definition}
A \emph{terminal valuation} of $X$ is a valuation defined by an 
exceptional divisor on a minimal model $f \colon Y \to X$ over $X$. 
\end{definition}

A minimal model $f \colon Y \to X$ over $X$ is, by definition, the outcome
of a minimal model program over $X$ started from any resolution of singularities of $X$.
It is characterized by two
properties: $Y$ has terminal singularities and $K_Y$ is relatively nef over $X$. 
The minimal resolution of a surface is the unique minimal model over it, 
and hence a valuation on a surface is terminal
if and only if it is essential. In higher dimensions, a variety $X$ can admit several 
relative minimal models over itself, but all of them are isomorphic in codimension one,
and therefore the set of terminal valuations of $X$ is determined by the set of
the exceptional divisors of any one of them. 

The following theorem is the natural generalization of Theorem~\ref{t:Nash-dim=2}
to higher dimensions.
The proof is similar to the proof of Theorem~\ref{t:Nash-dim=2} given in this paper
(with some technical adjustments needed to take into account the dimension of $X$), 
so we omit it.

\begin{theorem}[de Fernex and Docampo \protect{\cite[Theorem~1.1]{dFD16}}]
\label{t:dFD}
Every terminal valuation of $X$ is a Nash valuation. 
\end{theorem}

One way of thinking about this result is to contrast it to 
Proposition~\ref{p:Nash-implies-essential}. While the propostion
gives a necessary condition to be a Nash valuation, 
the theorem provides a sufficient condition, thus
squeezing the set of Nash valuations from the other side. 

\begin{remark}
There are no terminal valuations on a variety with terminal singularities,
nor over a variety which admits a small resolution, 
simply because in both case a minimal model over the variety does not extract any divisor.
The above result sheds no light on Nash valuations over varieties with such singularities. 
\end{remark}

There is another (more elementary) sufficient condition to be a Nash valuation.
We define a partial order among divisorial valuations on $X$
as follows. Given two divisorial valuations $v$ and $v'$ on $X$ 
we write $v \le v'$ if $c_X(v) \supset c_X(v')$
and $v(h) \le v'(h)$ for every $h \in \O_{X,c_X(v')}$. 
If moreover $v \ne v'$, then we write $v < v'$. 

\begin{definition}
A divisorial valuation $v$ centered in the singular locus of $X$ is said to be a
\emph{minimal valuation} of $X$ if it is minimal (with respect to the above 
partial order) among all divisorial valuations centered in $X_\sing$.
\end{definition}
 
\begin{proposition}
\label{p:minimal-implies-Nash}
Every minimal valuation of $X$ is a Nash valuation. 
\end{proposition}

\begin{proof}
Let $\val_E$ be a minimal valuation of $X$, and let $C_X(E) \subset X_\infty$
the associated maximal divisorial set. As $\val_E$ is centered in the singular locus of $X$, 
we have $C_X(E) \subset \p_X^{-1}(X_\sing)$. Let $C$ be an irreducible component of 
$\p_X^{-1}(X_\sing)$ containing $C_X(E)$. 
Let $\a \in C_X(E)$ and $\b \in C$ be the respective generic points, 
so that $\val_\a = \val_E$ and $\val_\b = \val_C$.
Since $\a$ is a specialization of $\b$, we have $\val_\b \le \val_\a$.
The hypothesis that $\val_E$ is minimal implies that $\val_E = \val_C$. 
Therefore $\val_E$ is a Nash valuation.
\end{proof}

Toric varieties provide another important class of varieties where the Nash map is surjective.

\begin{theorem}[Ishii and Koll\'ar \protect{\cite[Theorem~3.16]{IK03}}]
\label{t:Nash-toric}
For a divisorial valuation $v$ centered in the singular locus of 
a toric variety $X$ the following properties are equivalent:
\begin{enumerate}
\item
$v$ is a minimal valuation, 
\item
$v$ is a Nash valuation,
\item
$v$ is an essential valuation.
\end{enumerate}
In particular, the Nash map is surjective for every toric variety.
\end{theorem}

\begin{proof}
We already know that (a)~$\Rightarrow$~(b)~$\Rightarrow$~(c) by Propositions~\ref{p:minimal-implies-Nash} 
and~\ref{p:Nash-implies-essential}. We are left to prove that (c)~$\Rightarrow$~(a).

Without loss of generality, we may assume that $X = X(\s)$ is the affine toric variety
associated to a cone $\s \subset N \otimes \R$, where $N$ is the lattice dual to the character lattice $M$ 
of the torus. 
The elements of $\s \cap N$ are in bijection with the torus-invariant valuations on $X$. 
Let $\s_\sing := \bigcup_\t \t^\o$, where $\t$ ranges over all singular faces of $\s$
and $\t^\o$ denotes the relative interior of $\t$.\footnote{A rational polyhedral
cone $\t \subset N \otimes \R$ is
\emph{regular} if the primitive elements in the rays of $\t$ form a part of a basis of $N$, 
and is \emph{singular} otherwise.} The elements in $\s_\sing \cap N$ 
are in bijection with the torus-invariant valuations on $X$ centered in $X_\sing$.
Given two vectors $v,v' \in \s \cap N$, we write $v \le_\s v'$ if $v' \in v + \s$. 
If moreover $v \ne v'$, then we write $v <_\s v'$. 
This defines a partial order on $\s \cap N$, and hence on $\s_\sing \cap N$.

Since $X$ admits a torus-invariant resolution of singularities $f \colon Y \to X$
such that $f^{-1}(X_\sing)$ is the union of torus-invariant divisors, 
every essential valuation of $X$ is torus-invariant
and, as such, it corresponds to an element in $\s_\sing \cap N$. 

Let $v$ be a toric valuation centered in $X_\sing$, 
and assume that $v$ is not a minimal valuation of $X$. 
Then there is a divisorial valuation $\val_F$ on $X$
with center $c_X(F)$ contained in $X_\sing$ and containing the center of $v$, 
such that $\val_F < v$. Since $c_X(F) \subset X_\sing$, we have
$C_X(F) \subset \p_X^{-1}(X_\sing)$. Let $C'$ be an irreducible component of 
$\p_X^{-1}(X_\sing)$ containing $C_X(F)$, and let $v' = \val_{C'}$.
Note that $v'$ is torus-invariant, since, being a Nash valuation, it is essential. 
We identify $v$ and $v'$ with the corresponding elements in $\s_\sing \cap N$. 
We have $v' \le \val_F$ by the inclusion
$C_X(F) \subset C'$. It follows that $v' < v$, and hence $v' <_\s v$. 

We have $v = v' + v''$ for some $v'' \in \s \cap N \setminus \{0\}$. 
Let $\t$ be the 2-dimensional cone spanned by $v'$ and $v''$, and let $\G$
be the fan spanned by all elements in $\t \cap N$ that are minimal with respect
to the partial order $\le_\t$. Geometrically, $X(\t)$ is a surface singularity
and $X(\G)$ is its minimal resolution. Since $v$ is not minimal in 
$\t \cap N$ with respect to $\le_\t$, it must belong to the interior of a 2-dimensional face
of $\G$. 

For a suitable choice of vectors $v'$ and $v''$ adding up to $v$, 
the subdivision $\G$ of $\t$ can be extended to a subdivision $\D$ of $\s$
so that the faces of $\G$ are faces in $\D$ and $g \colon X(\D) \to X(\s)$ is a resolution
of singularities such that $g^{-1}(X_\sing)$ is a union of divisors
and $g$ is an isomorphism over the nonsingular locus of $X(\s)$ 
(we refer to the proof of \cite[Lemma~3.15]{IK03} for more details). 
By construction, $v$ does not belong to any ray in $\D$. This means that the
center of $v$ in $X(\D)$ is not a divisor, and therefore it
cannot be an irreducible component of $g^{-1}(X_\sing)$. 
We conclude that $v$ is not an essential valuation. 
\end{proof}

There are other examples where the Nash map is known to be surjective.
Theorem~\ref{t:Nash-toric} is extended to non-normal pretoric varieties in \cite{Ish05}, 
and locally analytically pretoric singularities in \cite{Ish06}, which implies
in particular that the Nash map is surjective for quasi-ordinary singularities
since any such singularity decomposes, locally analytically, into
analytically irreducible quasi-ordinary singularities, which are analytically pretoric. 
This result was further extended in \cite{GP07} to cover the case of reduced germs of quasi-ordinary
hypersurface singularities. More examples were discovered in \cite{PPP08,LJR12,LA16}. 

The following recent result provides yet another class of examples.

\begin{theorem}[Docampo and Nigro \protect{\cite[Proposition~11.2]{DN}}]
\label{t:DN}
The Nash map is surjective for Schubert varieties in Grassmannians. 
\end{theorem}

Reviewing the proof of this theorem would require setting up some notation
about Schubert varieties which we prefer to avoid here, 
so we limit ourselves to outline the argument. 
If $X$ is a Schubert variety and $X_\sing = \bigcup Z_i$ is the decomponsition
of the singular locus of $X$ into irreducible components, then there is a resolution 
$f \colon Y \to X$ such that $f^{-1}(Z_i)$ is irreducible for every $i$, so that
\[
f^{-1}(X_\sing)_\red = \bigcup f^{-1}(Z_i)
\]
is the decomposition into irreducible components of $f^{-1}(X_\sing)$.
The construction of the resolution is well-known to the experts; it appears for example in \cite{Zel83}.
This fact immediately implies that
\[
\p_X^{-1}(X_\sing)_\red = \bigcup \p_X^{-1}(Z_i)
\]
is the decomposition into irreducible components of $\p_X^{-1}(X_\sing)$.
Therefore, if $E_i$ is the divisor dominating $Z_i$
in the blow-up of $Y$ along $Z_i$, then $\val_{E_i}$ is a Nash valuation. 
Note that $\val_{E_i}$ is also both minimal and essential. 
In particular, the Nash map is surjective and the sets of minimal valuations, 
Nash valuations, and essential valuations are all the same.  

\begin{remark}
It was noticed in \cite{Zel83} that all Schubert varieties (in Grassmannians) admit small resolutions. 
This implies that they do not have any terminal valuations.
\end{remark}

Much of the work on the Nash problem 
for surfaces that preceded \cite{FdBPP12} has been based on explicit 
analysis of the singularities, thus resulting in more or less
explicit descriptions of spaces of arcs. 
Arc spaces have also been studied on 
varieties which possess additional structure such as a group action or a combinatorial nature, 
not just in relation to the Nash problem but also for its own sake.

An example is the paper \cite{DN} we just mentioned, 
whose main purpose is certainly not to solve the Nash
problem for Schubert varieties which mainly relies upon 
the existence of such nice resolutions,
but rather to describe the space of arcs of the Grassmannian which is done by
means of a decomposition of the arc space
that resembles the Schubert cell decomposition of the Grassmannian itself.
As it is explained at the end of \cite{DN}, these general results already
give a way of understanding directly
the decomposition of $\p_X^{-1}(X_\sing)$ into irreducible components
for any Schubert subvariety $X$ of the Grassmannian that does not make use
of any explicit resolution of singularities.
Arc spaces were previously studied for toric varieties in \cite{Ish04}
and for determinantal varieties in \cite{Doc13}. 

\medskip

The next corollary summarizes the general results on the Nash problem stated in this section.

\begin{corollary}
\label{c:inclusions}
For any variety $X$, there are inclusions
\[
\big(\{\,\text{minimal val's}\,\} \cup
\{\,\text{terminal val's}\,\}\big) \subset
\{\,\text{Nash val's}\,\} \subset
\{\,\text{essential val's}\,\}.
\]
\end{corollary}

\begin{proof}
The first inclusion follows by Theorem~\ref{t:dFD} and
Proposition~\ref{p:minimal-implies-Nash}. The second inclusion is the Nash map, which is defined by
Proposition~\ref{p:Nash-implies-essential}.
\end{proof}

Surfaces, toric varieties, and Schubert varieties (in Grassmannians) form three classes of varieties
for which the Nash map is known to be a bijection (Theorems~\ref{t:Nash-dim=2},~\ref{t:Nash-toric}
and~\ref{t:DN}). However, the surjectivity of the Nash map in these cases appears to hold for 
different reasons. 

For surfaces, the surjectivity follows by the fact that
every essential valuation is a terminal valuation. This means that, for surfaces, we have
\[
\{\,\text{minimal val's}\,\} \subset
\{\,\text{terminal val's}\,\} = 
\{\,\text{Nash val's}\,\} =
\{\,\text{essential val's}\,\},
\]
and the first inclusion may be strict. 

By contrast, for toric varieties and Schubert varieties,  
the surjectivity follows by the fact that
every essential valuation is a minimal valuation, which gives
\[
\{\,\text{terminal val's}\,\} \subset
\{\,\text{minimal val's}\,\} =
\{\,\text{Nash val's}\,\} =
\{\,\text{essential val's}\,\}, 
\]
and the first inclusion may be strict.

We do not know any example where the first inclusion in Corollary~\ref{c:inclusions} is a strict inclusion.
If the inclusion turned out to be always an equality, this would provide
a complete solution to the Nash problem. It is however quite possible that
the inclusion is strict in general. 

\begin{remark}
It would be interesting to see
whether there are other classes of varieties, like those listed above, 
for which all Nash valuations are either minimal or terminal, and to find examples
where this property holds but neither the set of minimal valuations
nor the set of terminal valuations suffices, alone, to exhaust all Nash valuations. 
\end{remark}

Approaching the set of Nash valuations from the other side, one could try to fix the Nash problem 
by modifying the definition of essential valuation. The definition given in \cite{Nas95}
requires testing a certain condition on the centers of
the valuations on all resolutions, but such definition is not restrictive enough to
characterize Nash valuations.
By enlarging the class of models where the 
condition is tested, one may hope to get a more restrictive version of essential valuations
that agrees with Nash valuations. 

This idea is explored in \cite{JK13} for valuations on a normal threefold $X$. 
The focus is on valuations centered at closed points in $X$. 

\begin{definition}
An isolated threefold singularity $Q \in Y$ is \emph{arc-wise Nash-trivial}
if for every general arc $\a \colon \Spec k[[t]] \to Y$ passing 
through a singular point $Q \in Y$ there is a morphism $\Phi \colon \Spec k[[s,t]] \to Y$
such that $\a(t) = \Phi(0,t)$ and $\Phi^{-1}(Q)$ is zero-dimensional.
\end{definition}

\begin{definition}
A divisorial valuation $v$ centered at a closed point $P$ of a normal threefold $X$ is a
\emph{very essential valuation} for $(P \in X)$ if for every
proper birational model $f \colon Y \to X$ where $Y$ has only 
isolated, $\Q$-factorial, {arc-wise Nash-trivial} singularities, 
either the center $c_Y(v)$ is an irreducible component of $f^{-1}(P)$,
or the maximal divisorial set $C_Y(v)$ is an irreducible component
of $\p_Y^{-1}(Q)$ for some singular point $Q \in Y$ such that 
$\dim_Q(f^{-1}(P)) \le 1$.\footnote{It is suggested in \cite{JK13} that, in this definition,
one may need to allow $Y$ to be an algebraic space.}
\end{definition}

The set of valuations that are very essential for a normal threefold singularity
$(P \in X)$ clearly contains
the set of valuations defined by the irreducible components of $\p_X^{-1}(P)$.
The converse is not known.

\begin{question}[Johnson--Koll\'ar \protect{\cite[Problem~38]{JK13}}]
Let $P$ be a closed point of a normal threefold $X$. 
Is every very essential valuation for $(P \in X)$ defined by an irreducible component of $\p_X^{-1}(P)$?
\end{question}

One can also consider the following alternative approach which is perhaps too optimistic
but it is easier to state in all dimensions.
Let $X$ be any variety. 

\begin{definition}
A proper birational morphism $f \colon Y \to X$
is an \emph{arc-wise semi-resolution} if for every irreducible component $E$ of $f^{-1}(X_\sing)$, 
the set $\p_Y^{-1}(E)$ is irreducible. 
\end{definition}

\begin{definition}
A divisorial valuation on $X$ is \emph{strongly essential} if its center  
on any arc-wise semi-resolution $f \colon Y \to X$ is an irreducible component of $f^{-1}(X_\sing)$.
\end{definition}

Clearly, strongly essential valuations are essential, 
and the same argument proving that Nash valuations are essential shows that
they are strongly essential. 

\begin{question}
Is a valuation on a variety $X$ a Nash valuation if and only if it is
a strongly essential valuation?
\end{question}

\section{The Nash problem in the analytic topology}

The space of arcs can be defined for any complex analytic variety $V$. 
If $V$ is defined by the vanishing of finitely many holomorphic 
functions $f_i(x) = 0$ in an analytic domain $U \subset \C^n$, then $V_\infty$
is defined as the set of $n$-ples of power series $x(t) \in \C[[t]]^n$ such that
$f_i(x(t)) = 0$ for all $i$. The jet spaces $V_m$ are defined similarly, and $V_\infty$
is their inverse limit. As such, it inherits the inverse limit analytic topology. 

Suppose that $V = X^\an$ is the analytification of some complex algebraic variety $X$. 
Then the points of $V_m$ are in bijection with $X_m(\C)$ and the points of $V_\infty$ with $X_\infty(\C)$. 
By Theorem~\ref{t:Gre}, the image of $\p_V^{-1}(V_\sing)$ in any finite level $V_m$ is a constructible set
and the decomposition of its closure into irreducible analytic subvarieties of $V_m$ 
stabilizes for $m \gg 1$. It follows that the truncation maps $V_{m+1} \to V_m$
establish a one-to-one correspondence between such decompositions. 
By passing to the limit as $m \to \infty$, 
one obtains a decomposition of $\p_V^{-1}(V_\sing)$ into finitely many families of arcs
which agrees with the decomposition into irreducible components of $\p_X^{-1}(X_\sing)$
given in Theorem~\ref{t:Nash}.

We can define essential valuations over $V$ analogously to the algebraic setting, by looking at the 
centers of the valuation on all analytic resolutions $W \to V$. It turns out that
the notion of essential valuation depends on the category. 

\begin{theorem}[de Fernex \protect{\cite[Theorem~5.1]{dF13}}]
There is a divisorial valuation over a complex threefold $X$ that is
essential in the category of schemes but not in the analytic category.
\end{theorem}

The example is a threefold $X \subset \C^4$ with an isolated singularity $O$. 
In this example, $\p_X^{-1}(O)$ is irreducible. 
Blowing up the singular point produces a variety $Y$ with an ordinary double point $P$, 
and the exceptional divisor $F$ defines the only Nash valuation on $X$.  
A resolution of $X$ is obtained by blowing up $P \in Y$. The divisor $G$ extracted by this second blow-up
defines a valuation $\val_G$ over $X$ which is essential in the category of schemes. 
This valuation is however not essential 
in the analytic category. This is due to the fact that $Y^\an$ admits a small analytic resolution
$W \to Y^\an$, which is not defined in the category of schemes, 
where the center of $\val_G$ is not an irreducible component
of the inverse image of $(X^\an)_\sing$. The difference in the notion of essential
valuation is reflected in this example in the fact that $X$ is locally $\Q$-factorial 
in the Zariski topology but not in the analytic topology.

One can also formulate the Nash problem in the category of algebraic spaces.
Just like in the analytic setting, 
there are more resolutions in the category of algebraic spaces than the 
category of schemes, and this can affect the notion of essential valuation.

\section{The Nash problem in positive characteristics}

Two ingredients play a pivotal role in the treatment of the subject
in characteristic zero: resolution of singularities and generic smoothness. 
Both are used in the proofs of Theorems~\ref{t:Kolchin} 
and~\ref{t:Nash} given here.

Kolchin's original proof of Theorems~\ref{t:Kolchin}
does not use resolution of singularities and works over any field
of characteristic zero. A geometric proof in this generality
is given in \cite[Theorem~3.6]{NS10}. 
Irreducibility fails however in positive characteristics. 
The following example was suggested by J\'anos Koll\'ar.\footnote{This example 
is discussed in \cite[Example~2.13]{IK03} in 
regard to the decomposition of $\p_X^{-1}(X_\sing)$ (cf.\ Example~\ref{eg:Nash-char-p}).}

\begin{example}
\label{eg:Kolchin-char-p}
Consider the $p$-fold Withney umbrella $X = (xy^p = z^p) \subset \A^3$ over a 
field $k$ of characteristic $p > 0$. 
The singular locus is the $x$-axis $X_\sing = (y=z=0)$. 
The ideal in $k[x,y,z,x',y',z']$ of the first jet scheme $X_1 \subset (\A^3)_1$
is generated by $xy^p - z^p$ and $x'y^p$, and its
primary decomposition is $(xy^p - z^p,x') \cap (y^p,z^p)$.
Therefore $X_1$ has two irreducible components: $V(xy^p - z^p,x')$ and $V(y,z)$. 
The first component is the closure of the image of $(X \setminus X_\sing)_\infty$.
On the other hand, the arc $\a = (t,0,0) \in (X_\sing)_\infty$ 
maps to a 1-jet that does not belong to such component. 
It follows that $(X_\sing)_\infty$ is not contained in
the closure of $(X \setminus X_\sing)_\infty$, and hence
$X_\infty$ is not irreducible. 
\end{example}

The next theorem tells us that
this is the worst that can happen, for arc spaces of varieties over perfect fields.

\begin{theorem}[Reguera \protect{\cite[Theorem~2.9]{Reg09}}]
\label{t:Reguera-irred}
The arc space $X_\infty$ of a variety $X$ defined over a perfect field $k$ has
a finite number of irreducible components only one of which is not contained in $(X_\sing)_\infty$. 
\end{theorem}

An example where irreducibility fails for a regular variety defined over a non-perfect field
can be found in \cite[Theorem~3.19]{NS10}.

The Nash problem was discussed for varieties in positive characteristics in \cite{IK03}.
For the reminder of the section, we shall assume that
\begin{quote}
\emph{$X$ is a variety
over an algebraically closed field $k$ of positive characteristic.}
\end{quote}
We write the decomposition of $\p_X^{-1}(X_\sing)_\red$ into irreducible components
as follows:
\[
\p_X^{-1}(X_\sing)_\red = \Big(\bigcup_{i \in I}C_i\Big) \cup \Big(\bigcup_{j \in J}D_j\Big),
\]
where $C_i \not\subset (X_\sing)_\infty$ and $D_j \subset (X_\sing)_\infty$. 

\begin{definition}
The components $C_i$ are called the \emph{good components} of $\p_X^{-1}(X_\sing)$.
\end{definition}

In order to gain some control on the decomposition, one needs to assume something 
about resolution of singularities. Since we cannot rely on generic smoothness, 
we consider resolutions that are isomorphisms over the smooth locus. 

\begin{definition}
A divisorial valuation $v$ on $X$ is \emph{essential} if for any 
resolution $f \colon Y \to X$ inducing an isomorphism over $X_\sm$, the center of $v$ on $Y$
is an irreducible component of $f^{-1}(X_\sing)$. 
\end{definition}

\begin{theorem}[Ishii--Koll\'ar \protect{\cite[Theorem~2.15]{IK03}}]
\label{t:IK-char-p}
Assume that $X$ has a resolution $f \colon Y \to X$ that is an isomorphism over $X_\sm$. 
Then for any good component $C_i$ of $\p_X^{-1}(X_\sing)$ there is
a divisor $E_i$ over $X$ such that $C_X(E_i) = C_i$. 
Moreover, the center of $\val_{E_i}$ on 
any such resolution $f$ is an irreducible component of $f^{-1}(X_\sing)$. 
In particular, $\p_X^{-1}(X_\sing)$ has only finitely many good components, 
and the valuation $\val_{C_i}$ associated to any such component is equal to $\val_{E_i}$
and hence is essential. 
\end{theorem}

\begin{definition}
We say that a divisorial valuation $v$ on $X$ is a 
\emph{Nash valuation} if $v = \val_{C_i}$ for some good component $C_i$ of $\p_X^{-1}(X_\sing)$.
\end{definition}

The theorem implies that every Nash valuation is essential. Just like in characteristic
zero, one can formulate the Nash problem by asking for which varieties 
Nash valuations are the same as essential valuations. 

In the terminology introduced in Definition~\ref{d:fat-thin}, Theorem~\ref{t:IK-char-p} implies 
that the $C_i$ are the {fat} components of $\p_X^{-1}(X_\sing)$.
The $D_j$, instead, are the {thin} components because they are contained in $(X_\sing)_\infty$.

The argument in the proof of Theorem~\ref{t:Nash} showing that $f_\infty$ induces a dominant
map from $\p_Y^{-1}(f^{-1}(X_\sing))$ to $\p_X^{-1}(X_\sing)$ breaks down in positive characteristics, 
and thin components can actually occur in the decomposition of $\p_X^{-1}(X_\sing)$.

\begin{example}
\label{eg:Nash-char-p}
Let $X = (xy^p = z^p) \subset \A^3$ defined over a 
field $k$ of characteristic $p$ as in Example~\ref{eg:Kolchin-char-p}.
Since $(X_\sing)_\infty$ is an irreducible component of $X_\infty$, 
it is also an irreducible component of $\p_X^{-1}(X_\sing)$. 
In particular, it is a thin component of this set. 
For a different argument which looks at the normalization of $X$, see \cite[Example~2.13]{IK03}.
\end{example}

\begin{corollary}
If we assume the existence of a resolution of singularities of $X$, then 
$\p_X^{-1}(X_\sing)$ has finitely many irreducible components.
\end{corollary}

\begin{proof}
We already know by Theorem~\ref{t:IK-char-p} that there are finitely many good components $C_i$, 
and the question is whether the number of thin components $D_j$ is finite. 
Since $D_j \subset (X_\sing)_\infty \subset \p_X^{-1}(X_\sing)$, each $D_j$ is
an irreducible component of $(X_\sing)_\infty$. Therefore it suffices to
check that $(X_\sing)_\infty$ has only finitely many irreducible components. 
Since $(X_\sing)_\infty$ is, set-theoretically, the union of the
arc spaces of the irreducible components of $X_\sing$, this property
follows from Theorem~\ref{t:Reguera-irred}.
\end{proof}

A possible way that may
get rid of the thin components in the decomposition of $\p_X^{-1}(X_\sing)$
is to restrict the attention to the main component $X_\infty^\main$
of $X_\infty$, namely, the irreducible component that dominates $X$. 
Let 
\[
\p_X^\main \colon X_\infty^\main \to X
\]
denote the restriction of $\p_X$. 

\begin{question}
Is every irreducible component of $(\p_X^\main)^{-1}(X_\sing)$ a good component of $\p_X^{-1}(X_\sing)$?
\end{question}

Resolutions which induce isomorphisms over the smooth locus are known to exist in positive characteristics
for surfaces and toric varieties.
It is therefore natural to consider the Nash problem for these classes of varieties. 

Regarding toric varieties, the proof of Theorem~\ref{t:Nash-toric}
is characteristic free, and therefore the statement holds in all characteristics. 

As for the case of surfaces, both proofs of Theorem~\ref{t:Nash-dim=2}
(the original one from \cite{FdBPP12} and the one 
based on \cite{dFD16}) use characteristic zero in an essential way. 
This is clear for the original proof where the problem is translated into
a topological problem. Most of the other proof (which is the one given here)
is characteristic free, but inequality~(4) relies on a computation which fails
if the map $\f'$ is wildly ramified at the generic point of some of the divisors $C_i$. 

This leaves the following question open.

\begin{question}
For a surface defined over an algebraically closed field of
positive characteristic, is every valuation associated to an exceptional divisor 
on the minimal resolution a Nash valuation?
\end{question}

Some cases are known, for instance the case of sandwiches singularities \cite{LJR99}.

\bibliographystyle{amsplain}

\begin{bibdiv}
\begin{biblist}

\bib{Bha}{article}{
   author={Bhatt, Bhargav},
   title={Algebraization and Tannaka duality},
   note={Preprint, {\tt arXiv:1404.7483}},
}

\bib{Bat99}{article}{
   author={Batyrev, Victor V.},
   title={Non-Archimedean integrals and stringy Euler numbers of
   log-terminal pairs},
   journal={J. Eur. Math. Soc. (JEMS)},
   volume={1},
   date={1999},
   number={1},
   pages={5--33},
}

\bib{dF13}{article}{
   author={de Fernex, Tommaso},
   title={Three-dimensional counter-examples to the Nash problem},
   journal={Compos. Math.},
   volume={149},
   date={2013},
   number={9},
   pages={1519--1534},
}

\bib{dFD16}{article}{
   author={de Fernex, Tommaso},
   author={Docampo, Roi},
   title={Terminal valuations and the Nash problem},
   journal={Invent. Math.},
   volume={203},
   date={2016},
   number={1},
   pages={303--331},
}

\bib{dFEI08}{article}{
   author={de Fernex, Tommaso},
   author={Ein, Lawrence},
   author={Ishii, Shihoko},
   title={Divisorial valuations via arcs},
   journal={Publ. Res. Inst. Math. Sci.},
   volume={44},
   date={2008},
   number={2},
   pages={425--448},
}

\bib{DL99}{article}{
   author={Denef, Jan},
   author={Loeser, Fran{\c{c}}ois},
   title={Germs of arcs on singular algebraic varieties and motivic
   integration},
   journal={Invent. Math.},
   volume={135},
   date={1999},
   number={1},
   pages={201--232},
}

\bib{Doc13}{article}{
   author={Docampo, Roi},
   title={Arcs on determinantal varieties},
   journal={Trans. Amer. Math. Soc.},
   volume={365},
   date={2013},
   number={5},
   pages={2241--2269},
}

\bib{DN}{article}{
   author={Docampo, Roi},
   author={Nigro, Antonio},
   title={The arc space of the Grassmannian},
   note={Preprint, {\tt arXiv: 1510.08833}},
}

\bib{Dri}{article}{
   author={Drinfeld, Vladimir},
   title={On the Grinberg--Kazhdan formal arc theorem},
   note={Preprint, {\tt arXiv: math.AG/ 0203263}},
}

\bib{ELM04}{article}{
   author={Ein, Lawrence},
   author={Lazarsfeld, Robert},
   author={Musta{\c{t}}{\v{a}}, Mircea},
   title={Contact loci in arc spaces},
   journal={Compos. Math.},
   volume={140},
   date={2004},
   number={5},
   pages={1229--1244},
}

\bib{EM04}{article}{
   author={Ein, Lawrence},
   author={Musta{\c{t}}{\u{a}}, Mircea},
   title={Inversion of adjunction for local complete intersection varieties},
   journal={Amer. J. Math.},
   volume={126},
   date={2004},
   number={6},
   pages={1355--1365},
}

\bib{EM09}{article}{
   author={Ein, Lawrence},
   author={Musta{\c{t}}{\u{a}}, Mircea},
   title={Jet schemes and singularities},
   conference={
      title={Algebraic geometry---Seattle 2005. Part 2},
   },
   book={
      series={Proc. Sympos. Pure Math.},
      volume={80},
      publisher={Amer. Math. Soc.},
      place={Providence, RI},
   },
   date={2009},
   pages={505--546},
}

\bib{EMY03}{article}{
   author={Ein, Lawrence},
   author={Musta{\c{t}}{\u{a}}, Mircea},
   author={Yasuda, Takehiko},
   title={Jet schemes, log discrepancies and inversion of adjunction},
   journal={Invent. Math.},
   volume={153},
   date={2003},
   number={3},
   pages={519--535},
}

\bib{FdB12}{article}{
   author={Fern{\'a}ndez de Bobadilla, Javier},
   title={Nash problem for surface singularities is a topological problem},
   journal={Adv. Math.},
   volume={230},
   date={2012},
   number={1},
   pages={131--176},
}

\bib{FdBPP12}{article}{
   author={Fern{\'a}ndez de Bobadilla, Javier},
   author={Pe Pereira, Mar{\'{\i}}a},
   title={The Nash problem for surfaces},
   journal={Ann. of Math. (2)},
   volume={176},
   date={2012},
   number={3},
   pages={2003--2029},
}

\bib{Gil02}{article}{
   author={Gillet, Henri},
   title={Differential algebra---a scheme theory approach},
   conference={
      title={Differential algebra and related topics},
      address={Newark, NJ},
      date={2000},
   },
   book={
      publisher={World Sci. Publ., River Edge, NJ},
   },
   date={2002},
   pages={95--123},
}

\bib{GP07}{article}{
   author={Gonz{\'a}lez P{\'e}rez, P. D.},
   title={Bijectiveness of the Nash map for quasi-ordinary hypersurface
   singularities},
   journal={Int. Math. Res. Not. IMRN},
   date={2007},
   number={19},
}

\bib{Gre66}{article}{
   author={Greenberg, Marvin J.},
   title={Rational points in Henselian discrete valuation rings},
   journal={Inst. Hautes \'Etudes Sci. Publ. Math.},
   number={31},
   date={1966},
   pages={59--64},
}

\bib{GK00}{article}{
   author={Grinberg, M.},
   author={Kazhdan, D.},
   title={Versal deformations of formal arcs},
   journal={Geom. Funct. Anal.},
   volume={10},
   date={2000},
   number={3},
   pages={543--555},
}

\bib{EGAiii_1}{article}{
   author={Grothendieck, A.},
   title={\'El\'ements de g\'eom\'etrie alg\'ebrique. III. \'Etude
   cohomologique des faisceaux coh\'erents. I},
   language={French},
   journal={Inst. Hautes \'Etudes Sci. Publ. Math.},
   number={11},
   date={1961},
   pages={167},
}

\bib{EGAiv_3}{article}{
   author={Grothendieck, A.},
   title={\'El\'ements de g\'eom\'etrie alg\'ebrique. IV. \'Etude locale des
   sch\'emas et des morphismes de sch\'emas. III},
   journal={Inst. Hautes \'Etudes Sci. Publ. Math.},
   number={28},
   date={1966},
   pages={255},
}

\bib{Hir64}{article}{
   author={Hironaka, Heisuke},
   title={Resolution of singularities of an algebraic variety over a field
   of characteristic zero. I, II},
   journal={Ann. of Math. (2) {\bf 79} (1964), 109--203; ibid. (2)},
   volume={79},
   date={1964},
   pages={205--326},
}

\bib{Ish04}{article}{
   author={Ishii, Shihoko},
   title={The arc space of a toric variety},
   journal={J. Algebra},
   volume={278},
   date={2004},
   number={2},
   pages={666--683},
}

\bib{Ish05}{article}{
   author={Ishii, Shihoko},
   title={Arcs, valuations and the Nash map},
   journal={J. Reine Angew. Math.},
   volume={588},
   date={2005},
   pages={71--92},
}

\bib{Ish06}{article}{
   author={Ishii, Shihoko},
   title={The local Nash problem on arc families of singularities},
   language={English, with English and French summaries},
   journal={Ann. Inst. Fourier (Grenoble)},
   volume={56},
   date={2006},
   number={4},
   pages={1207--1224},
}

\bib{Ish08}{article}{
   author={Ishii, Shihoko},
   title={Maximal divisorial sets in arc spaces},
   conference={
      title={Algebraic geometry in East Asia---Hanoi 2005},
   },
   book={
      series={Adv. Stud. Pure Math.},
      volume={50},
      publisher={Math. Soc. Japan, Tokyo},
   },
   date={2008},
   pages={237--249},
}
				
\bib{IK03}{article}{
   author={Ishii, Shihoko},
   author={Koll{\'a}r, J{\'a}nos},
   title={The Nash problem on arc families of singularities},
   journal={Duke Math. J.},
   volume={120},
   date={2003},
   number={3},
   pages={601--620},
}

\bib{JK13}{article}{
   author={Johnson, Jennifer M.},
   author={Koll{\'a}r, J{\'a}nos},
   title={Arc spaces of $cA$-type singularities},
   journal={J. Singul.},
   volume={7},
   date={2013},
   pages={238--252},
}

\bib{Kol73}{book}{
   author={Kolchin, E. R.},
   title={Differential algebra and algebraic groups},
   note={Pure and Applied Mathematics, Vol. 54},
   publisher={Academic Press},
   place={New York},
   date={1973},
   pages={xviii+446},
}

\bib{KN15}{article}{
   author={Koll{\'a}r, J{\'a}nos},
   author={N{\'e}methi, Andr{\'a}s},
   title={Holomorphic arcs on singularities},
   journal={Invent. Math.},
   volume={200},
   date={2015},
   number={1},
   pages={97--147},
}

\bib{Kon95}{book}{
   author={Kontsevich, Maxim},
   title={String cohomology},
   note={Lecture at Orsay},
   year={1995},
}

\bib{LJ80}{article}{
   author={Lejeune-Jalabert, Monique},
   title={Arcs analytiques et r\'esolution minimale des surfaces quasihomog\`enes},
   language={French},
   conference={
      title={S\'eminaire sur les Singularit\'es des Surfaces},
      address={Palaiseau, France},
      date={1976/1977},
   },
   book={
      series={Lecture Notes in Math.},
      volume={777},
      publisher={Springer},
      place={Berlin},
   },
   date={1980},
   pages={303Ð-332},
}

\bib{LJR99}{article}{
   author={Lejeune-Jalabert, Monique},
   author={Reguera, Ana J.},
   title={Arcs and wedges on sandwiched surface singularities},
   journal={Amer. J. Math.},
   volume={121},
   date={1999},
   number={6},
   pages={1191--1213},
}

\bib{LJR12}{article}{
   author={Lejeune-Jalabert, Monique},
   author={Reguera, Ana J.},
   title={Exceptional divisors that are not uniruled belong to the image of
   the Nash map},
   journal={J. Inst. Math. Jussieu},
   volume={11},
   date={2012},
   number={2},
   pages={273--287},
}

\bib{LA16}{article}{
   author={Leyton-Alvarez, Maximiliano},
   title={Familles d'espaces de $m$-jets et d'espaces d'arcs},
   language={French, with English and French summaries},
   journal={J. Pure Appl. Algebra},
   volume={220},
   date={2016},
   number={1},
   pages={1--33},
}

\bib{Mus01}{article}{
   author={Musta{\c{t}}{\u{a}}, Mircea},
   title={Jet schemes of locally complete intersection canonical
   singularities},
   note={With an appendix by David Eisenbud and Edward Frenkel},
   journal={Invent. Math.},
   volume={145},
   date={2001},
   number={3},
   pages={397--424},
}

\bib{Mus02}{article}{
   author={Musta{\c{t}}{\u{a}}, Mircea},
   title={Singularities of pairs via jet schemes},
   journal={J. Amer. Math. Soc.},
   volume={15},
   date={2002},
   number={3},
   pages={599--615 (electronic)},
}
		
\bib{Nas95}{article}{
   author={Nash, John F., Jr.},
   title={Arc structure of singularities},
   note={A celebration of John F. Nash, Jr.},
   journal={Duke Math. J.},
   volume={81},
   date={1995},
   number={1},
   pages={31--38 (1996)},
}

\bib{NS10}{article}{
   author={Nicaise, Johannes},
   author={Sebag, Julien},
   title={Greenberg approximation and the geometry of arc spaces},
   journal={Comm. Algebra},
   volume={38},
   date={2010},
   number={11},
   pages={4077--4096},
}

\bib{Pas89}{article}{
   author={Pas, Johan},
   title={Uniform $p$-adic cell decomposition and local zeta functions},
   journal={J. Reine Angew. Math.},
   volume={399},
   date={1989},
   pages={137--172},
}

\bib{PP13}{article}{
   author={Pe Pereira, Mar{\'{\i}}a},
   title={Nash problem for quotient surface singularities},
   journal={J. Lond. Math. Soc. (2)},
   volume={87},
   date={2013},
   number={1},
   pages={177--203},
}

\bib{Ple05}{article}{
   author={Pl{\'e}nat, Camille},
   title={\`A propos du probl\`eme des arcs de Nash},
   language={French, with English and French summaries},
   journal={Ann. Inst. Fourier (Grenoble)},
   volume={55},
   date={2005},
   number={3},
   pages={805--823},
}

\bib{Ple08}{article}{
   author={Pl{\'e}nat, Camille},
   title={The Nash problem of arcs and the rational double points $D_n$},
   language={English, with English and French summaries},
   journal={Ann. Inst. Fourier (Grenoble)},
   volume={58},
   date={2008},
   number={7},
   pages={2249--2278},
}

\bib{PPP06}{article}{
   author={Pl{\'e}nat, Camille},
   author={Popescu-Pampu, Patrick},
   title={A class of non-rational surface singularities with bijective Nash
   map},
   language={English, with English and French summaries},
   journal={Bull. Soc. Math. France},
   volume={134},
   date={2006},
   number={3},
   pages={383--394},
}

\bib{PPP08}{article}{
   author={Pl{\'e}nat, Camille},
   author={Popescu-Pampu, Patrick},
   title={Families of higher dimensional germs with bijective Nash map},
   journal={Kodai Math. J.},
   volume={31},
   date={2008},
   number={2},
   pages={199--218},
}

\bib{PS12}{article}{
   author={Pl{\'e}nat, Camille},
   author={Spivakovsky, Mark},
   title={The Nash problem of arcs and the rational double point ${\rm E}_6$},
   journal={Kodai Math. J.},
   volume={35},
   date={2012},
   number={1},
   pages={173--213},
}

\bib{PS15}{article}{
   author={Pl{\'e}nat, Camille},
   author={Spivakovsky, Mark},
   title={The Nash problem and its solution: a survey},
   journal={J. Singul.},
   volume={13},
   date={2015},
   pages={229--244},
}

\bib{Reg95}{article}{
   author={Reguera, Ana J.},
   title={Families of arcs on rational surface singularities},
   journal={Manuscripta Math.},
   volume={88},
   date={1995},
   number={3},
   pages={321--333},
}

\bib{Reg06}{article}{
   author={Reguera, Ana J.},
   title={A curve selection lemma in spaces of arcs and the image of the
   Nash map},
   journal={Compos. Math.},
   volume={142},
   date={2006},
   number={1},
   pages={119--130},
}

\bib{Reg09}{article}{
   author={Reguera, Ana J.},
   title={Towards the singular locus of the space of arcs},
   journal={Amer. J. Math.},
   volume={131},
   date={2009},
   number={2},
   pages={313--350},
}

\bib{Reg12}{article}{
   author={Reguera, Ana J.},
   title={Arcs and wedges on rational surface singularities},
   journal={J. Algebra},
   volume={366},
   date={2012},
   pages={126--164},
}

\bib{Voj07}{article}{
   author={Vojta, Paul},
   title={Jets via Hasse-Schmidt derivations},
   conference={
      title={Diophantine geometry},
   },
   book={
      series={CRM Series},
      volume={4},
      publisher={Ed. Norm., Pisa},
   },
   date={2007},
   pages={335--361},
}

\bib{Zel83}{article}{
   author={Zelevinski{\u\i}, A. V.},
   title={Small resolutions of singularities of Schubert varieties},
   language={Russian},
   journal={Funktsional. Anal. i Prilozhen.},
   volume={17},
   date={1983},
   number={2},
   pages={75--77},
}





\end{biblist}
\end{bibdiv}

\end{document}